\documentclass{amsart}
\usepackage{amsmath}
\usepackage{amssymb}
\usepackage{amsfonts}
\usepackage{float}
\usepackage{fancyhdr}
\usepackage[noadjust]{cite}

\newtheorem{theorem}{Theorem}[section]
\theoremstyle{plain}

\newtheorem{notation}{Notation}

\newtheorem{corollary}[theorem]{Corollary}
\newtheorem{lemma}[theorem]{Lemma}
\newtheorem{definition}[theorem]{Definition}
\newtheorem{proposition}[theorem]{Proposition}
\newtheorem{remark}[theorem]{Remark}
\newtheorem{example}[theorem]{Example}
\numberwithin{equation}{section}
\input{tcilatex}
\begin{document}
\title{Dioid Partitions of Groups}
\author{Ishay Haviv}
\address[Ishay Haviv]{The School of Computer Sciences, The Academic College
of Tel-Aviv-Yaffo, 2 Rabenu Yeruham St., Tel-Aviv 61083, Israel}
\author{Dan Levy}
\address[Dan Levy]{The School of Computer Sciences, The Academic College of
Tel-Aviv-Yaffo, 2 Rabenu Yeruham St., Tel-Aviv 61083, Israel}
\subjclass{ }
\date{\today }

\begin{abstract}
A partition of a group is \textit{a dioid partition} if the following three
conditions are met: The setwise product of any two parts is a union of
parts, there is a part that multiplies as an identity element, and the
inverse of a part is a part. This kind of a group partition was first
introduced by Tamaschke in 1968. We show that a dioid partition defines a
dioid structure over the group, analogously to the way a Schur ring over a
group is defined. After proving fundamental properties of dioid partitions,
we focus on three part dioid partitions of cyclic groups of prime order. We
provide classification results for their isomorphism types as well as for
the partitions themselves.
\end{abstract}

\maketitle

\section{Introduction \label{Sect_Intro}}

A \emph{dioid} is a triple $\left( D,\oplus ,\otimes \right) $, where $D$ is
a set and $\oplus $ and $\otimes $ are two binary operations over $D$, such
that $\left( D,\oplus ,\otimes \right) $ is a semiring and $D$ is
canonically ordered by $\oplus $ (see Definition \ref{Def_Dioid}). Dioids
and rings share most of their axioms, except that in the ring case, the
additive substructure is a commutative group (hence cannot be canonically
ordered by $\oplus $ - see Remark \ref{Rem_DioidsAndRingsDoNotIntersect}).
In addition to the inherent algebraic interest in a structure which is both
similar to and clearly distinct from a ring, dioids have attracted much
attention over the last 30 years due to their interesting practical
applications. These include the solution of a variety of optimal path
problems in graph theory, extensions of classical algorithms for solving
shortest paths problems with time constraints, data analysis techniques,
hierarchical clustering and preference analysis, algebraic modeling of
fuzziness and uncertainty, and discrete event systems in automation (see 
\cite%
{GondranMinouxDioidBook2010,CSX2014,BaldanGadducci2008,GondranMinouxFuzzy2007,DelCampoEtAl2015}%
).

As is well known, groups give rise in a natural way to rings via the
construction of group rings and its generalization to Schur rings. The aim
of the present paper is to point out a similar connection between groups and
dioids. We consider the following definition, first introduced by Tamaschke 
\cite{Tamaschke1968} (see Remark \ref{Rem_TMain}).

\begin{definition}
\label{Def_d-partition}A \emph{dioid partition} (d-partition for short) of a
group $G$ is a partition $\Pi $ of $G$ which satisfies the following:

\begin{enumerate}
\item[a.] The closure property: $\forall \pi _{1},\pi _{2}\in \Pi $, the
setwise product $\pi _{1}\pi _{2}$ is a union of parts of $\Pi $.

\item[b.] Existence of an identity: There exists $e\in \Pi $ satisfying $%
e\pi =\pi e=\pi $ for every $\pi \in \Pi $. We will also write $e_{\Pi }$.

\item[c.] The inverse property: $\Pi $ is invariant under the $G\rightarrow
G $ mapping defined by $g\longmapsto g^{-1}$ for all $g\in G$.
\end{enumerate}

When $e=\left\{ 1_{G}\right\} $ we shall say that $\Pi $ is a 1d-partition.
\end{definition}

Note that certain classical group partitions satisfy the above conditions
and thus form d-partitions.

\begin{example}
\label{Example_Classical}

1. If $A$ is a subgroup of the group $G$ then the set $\Pi =\left\{ AxA|x\in
G\right\} $ of all double cosets of $A$ in $G$ is a d-partition of $G$ with
identity $A$ and $\left( AxA\right) ^{-1}=Ax^{-1}A$. The case $A=\left\{
1_{G}\right\} $ yields the \emph{singleton partition} $\Pi =\left\{ \left\{
g\right\} |g\in G\right\} $ of $G$.

2. Let $\Pi $ be the set of all conjugacy classes of the group $G$. Here the
identity element is $Cl\left( 1_{G}\right) =\left\{ 1_{G}\right\} $, and $%
Cl\left( x\right) ^{-1}=Cl\left( x^{-1}\right) $.
\end{example}

The following theorem shows that a d-partition of a group $G$ gives rise to
a dioid.

\begin{theorem}
\label{Th_DioidOutOfd-partition}Let $G$ be a group and $\Pi $ a d-partition
of $G$. Let $D_{\Pi }$ be the set of all possible unions of parts of $\Pi $.
Denote set union by $\oplus $ and setwise product of subsets of $G$ by $%
\otimes $. Then $\left( D_{\Pi },\oplus ,\otimes \right) $ is a dioid.
\end{theorem}

A dioid which is constructed from a d-partition $\Pi $ of $G$ as in Theorem %
\ref{Th_DioidOutOfd-partition}, will be called a \emph{Schur dioid} over $G$
(induced by $\Pi $). The reason for this terminology is the strong
resemblance of Schur dioids to Schur rings (see Section \ref%
{Subsect_SAndDPartitions}). In fact, every Schur ring defines a Schur dioid
in a natural way (see Proposition \ref{Prop_s_Implies_d}). The study of
Schur rings (and the more general structures of association schemes and
coherent configurations) is a well-developed research area which goes back
to the work of Issai Schur \cite{Schur1933}. Its first systematic treatment
was carried by H. Wielandt \cite{WielandtPermGroups1964}, and since then it
has been further developed and found fruitful applications (see \cite%
{MuzychukKlinPoschel2001},\cite{MuzychukPonomarenko2009}). Our interest in
studying Schur dioids was motivated by an observation in \cite%
{CGLMS2014conjugates}.\footnote{%
In \cite{CGLMS2014conjugates} it is proved that any group with a $BN$-pair
and a finite Weyl group is the product of three conjugates of the Borel
subgroup $B$ \cite[Theorem 1.5]{CGLMS2014conjugates}. The calculations
needed for the proof take place in the Schur dioid induced by the double
cosets of $B$.}

In the first part of the paper we consider general properties of
d-partitions. We prove several basic results and discuss the analogies and
direct relations between Schur dioids and Schur rings. We then present
several constructions of d-partitions from other d-partitions. The following
theorem provides conditions for refining d-partitions and for making them
coarser, and shows that every d-partition can be refined to a 1d-partition.
For part (c) see Remark \ref{Rem_T1}.

\begin{theorem}
\label{Th_General_1d_connection}Let $G$ be a group. For a d-partition $\Pi $
of $G$ and $A\leq G$ set 
\begin{eqnarray*}
\Pi _{<A} &:&=\left\{ \pi \in \Pi |\pi A=A\pi =A\right\} ,~ \\
\Pi _{>A} &:&=\left\{ \pi \in \Pi |\pi A=A\pi =\pi \right\} \backslash
\left\{ A\right\} \text{.}
\end{eqnarray*}

\begin{enumerate}
\item[a.] Let $\Pi $ be a d-partition of $G$, let $A\leq G$, and let $\Pi
^{\prime }:=\Pi _{>A}\cup \left\{ A\right\} $. Then $\Pi ^{\prime }$ is a
d-partition of $G$ with $e_{\Pi ^{\prime }}=A$ if and only if $\left\{ \Pi
_{<A},\Pi _{>A}\right\} $ is a partition of $\Pi $.

\item[b.] Let $A\leq G$, let $\Pi _{A}$ be a d-partition of $A$, and let $%
\Pi ^{\prime }$ be a d-partition of $G$ with $e_{\Pi ^{\prime }}=A$. Then $%
\Pi :=\Pi _{A}\cup \left( \Pi ^{\prime }\backslash \left\{ A\right\} \right) 
$ is a d-partition of $G$ with $e_{\Pi }=e_{\Pi _{A}}$. In particular, for
any d-partition $\Pi ^{\prime }$ of $G$ with $e_{\Pi ^{\prime }}=A$ there
exists a 1d-partition $\Pi $ of $G$ such that $\Pi _{>A}^{\prime }\subseteq
\Pi $.

\item[c.] Let $\Pi $ be a d-partition of $G$, and let $A\leq G$ be such that 
$\Pi _{<A}$ is a partition of $A$. Then $\Pi _{<A}$ is a d-partition of $A$
with $e_{\Pi _{<A}}=e_{\Pi }$ and 
\begin{equation*}
\Pi ^{\prime }:=\left\{ A\pi A|\pi \in \Pi \backslash \Pi _{<A}\right\} \cup
\left\{ A\right\}
\end{equation*}%
is a d-partition of $G$ with $e_{\Pi ^{\prime }}=A$.

\item[d.] Let $\Pi $ be a d-partition of $G$, and let $A<G$ be such that $%
\Pi _{<A}$ is a partition of $A$. Set $\pi _{c}:=G\backslash A$. Then $\Pi
_{C}:=\Pi _{<A}\cup \left\{ \pi _{c}\right\} $ is a d-partition of $G$ with $%
e_{\Pi _{C}}=e_{\Pi _{<A}}$.
\end{enumerate}
\end{theorem}

The next theorem shows how automorphisms of $G$ can induce new, coarser,
d-partitions of $G$.

\begin{theorem}
\label{Th_CoarsenessViaAction}Let $G$ be a group and let $A$ be a group that
acts on $G$ as automorphisms. Let $\Pi $ be a d-partition of $G$. Suppose
that the action of $A$ on $G$ induces an action of $A$ on $\Pi $, namely,
that for any $a\in A$ and any $\pi \in \Pi $ we have $\pi ^{a}\in \Pi $.
Denote by $U_{\left[ \pi \right] }$ the union of the parts of $\Pi $
contained in the orbit of $\pi \in \Pi $ under the induced action, that is, $%
U_{\left[ \pi \right] }:=\tbigcup\limits_{a\in A}\pi ^{a}$, and set $\Pi
^{\prime }:=\left\{ U_{\left[ \pi \right] }|\pi \in \Pi \right\} $. Then $%
\Pi ^{\prime }$ is a d-partition of $G$, with $e_{\Pi ^{\prime }}=e_{\Pi }$.
\end{theorem}

One natural example where the action of $A$ on $G$ induces an action of $A$
on $\Pi $ is provided by taking $\Pi $ to be a double coset partition of an $%
A$-invariant subgroup of $G$. This idea is elaborated in Section \ref%
{Sect_Inducing} (see Corollary \ref{Coro_AInvariantH}, Example \ref%
{Example_ZnMultiplicative} and Corollary \ref{Coro_Supplement_construction}).

Finally we show that d-partitions lift from a quotient group to the group.

\begin{theorem}
\label{Th_LiftFromQuotient}Let $G$ be a group, $N\trianglelefteq G$, and $%
\overline{\Pi }$ a d-partition of $\overline{G}:=G/N$. Then the set $\Pi $
of all of the full preimages in $G$ of the parts of $\overline{\Pi }$ is a
d-partition of $G$.
\end{theorem}

In the second part of the paper we study d-partitions of finite cyclic
groups of prime order. While the Schur partitions of such groups were fully
classified by Gordon in \cite{BGordonPacific1964} (see Theorem \ref%
{Th_BGordon}), the possibilities for d-partitions are significantly richer.
We focus on the case of $3$-part d-partitions which turn out to be related
to natural questions in additive combinatorics. Although the $3$-part case
is still quite rich, we are able to provide the following classification
result.

\begin{theorem}
\label{Th_3PartIsoClassification}Let $p$ be a prime, and let $G=\left( 
\mathbb{Z}_{p},+_{p}\right) $. Then the $3$-part d-partitions of $G$ are all
the partitions $\Pi =\left\{ \pi _{0},\pi _{1},\pi _{2}\right\} $ of $G$
with $\pi _{0}=\left\{ 0\right\} $ and $\left\vert \pi _{1}\right\vert \leq
\left\vert \pi _{2}\right\vert $, for which exactly one of the following
holds:

\begin{enumerate}
\item[(1)] $p=3$ and $\Pi $ is the singleton partition.

\item[(2)] $p=5$ and $\Pi =\left\{ \left\{ 0\right\} ,\left\{ 1,4\right\}
,\left\{ 2,3\right\} \right\} $.

\item[(3)] $p>5$, $\pi _{1}=-\pi _{1}$ and $\pi _{1}+\pi _{1}=\pi _{0}\cup
\pi _{2}$.

\item[(4)] $p>5$, $\pi _{1}=-\pi _{1}$, $\pi _{1}+\pi _{1}=G$ and $\pi _{2}$
is not an arithmetic progression of size $\frac{p-1}{2}$.

\item[(5)] $p>5$, $\pi _{1}=-\pi _{2}$ and $\pi _{1}+\pi _{1}=\pi _{1}\cup
\pi _{2}$.
\end{enumerate}
\end{theorem}

We note that Theorem \ref{Th_3PartIsoClassification} classifies all of the
distinct isomorphism types of $3$-part d-partitions of $\mathbb{Z}_{p}$, as
each of the cases (1)-(5) defines a distinct isomorphism type.\footnote{%
Informally, an isomorphism between two d-partitions is a bijection between
the two partitions that respects all of the algebraic content of Definition %
\ref{Def_d-partition}. For the precise definition see Section \ref%
{Subsect_BasicPropOfD}, Lemma \ref{Lem_Def_DioidStructureConst} and
Definition \ref{Def_IsomorphicDPartitions}.} For cases (4) and (5) of
Theorem \ref{Th_3PartIsoClassification} we provide a full and explicit
characterization of the corresponding d-partitions and use this
characterization to estimate their numbers (see Sections \ref%
{Subsect_SolutionsTo2} and \ref{Subsect_SolutionsTo3}). In case (3) of the
theorem, $\pi _{1}$ is a symmetric, complete sum-free subset of $G$ (see
Section \ref{Subsect_SolutionsTo1}). Such sets received a considerable
amount of attention in the literature and found interesting applications
(e.g., \cite{Schur1916,CameronSumFreeSurvey1987,GreenRuzsa2005}).

\begin{remark}
\label{Rem_TMain}During the review process of the current paper, we have
learned from the referees that the notion of d-partitions was previously
studied by Tamaschke (see \cite{Tamaschke1968}). Tamaschke introduces the
definition of the object we call a d-partition $\Pi $ of a group $G$, as
part of his definition of an \emph{S-semigroup} $T$ on $G$ (which is the
semigroup generated by the elements of $\Pi $), and offers interesting
extensions of basic concepts in group theory to S-semigroups. He also
identifies and uses the dioid structure induced by d-partitions, however,
and perhaps not surprisingly, without referring to the term dioid or to its
general definition. Additional points of overlap between Tamaschke's
discussion and ours are given in Remarks \ref{Rem_T2}, \ref{Rem_T3}, \ref%
{Rem_T1}, \ref{Rem_T4}. Interestingly, our study of 3-part d-partitions of
cyclic groups allows us to settle a question raised by Tamaschke in \cite[%
Problem 1.19]{Tamaschke1969} (see Remark \ref%
{Rem_d-partitionsWhichAreNotS-partitions}). We believe that reinterpreting
Tamaschke's theory of S-semigroups from the point of view of dioids might be
of interest in future research.
\end{remark}

The paper is organized as follows. Section \ref{Sect_Background} summarizes
some background concepts and results. Section \ref{Sect_D-partitionsOfGroups}%
\ presents several results on d-partitions, including the proof of Theorem %
\ref{Th_DioidOutOfd-partition}, and analogies between Schur dioids and Schur
rings. Section \ref{Sect_Inducing} contains the proofs of theorems \ref%
{Th_General_1d_connection}, \ref{Th_CoarsenessViaAction} and \ref%
{Th_LiftFromQuotient}. Finally, Section \ref{Sect_DPartitonsOfZp} presents
our results on $3$-part d-partitions of cyclic groups of prime order,
including the proof of Theorem \ref{Th_3PartIsoClassification}.

\section{Background Concepts and Results \label{Sect_Background}}

\subsection{Semirings and Dioids\label{Subsect_SemiAndDioids}}

\begin{definition}[{\protect\cite[Chapter 1]{GondranMinouxDioidBook2010}}]
\label{Def_Dioid}A triple $\left( D,\oplus ,\otimes \right) $, where $D$ is
a set and $\oplus ,\otimes $ are two binary operations over $D$, is a \emph{%
semiring} if the following conditions (a) and (b) hold.

\begin{enumerate}
\item[\textbf{(a)}] $\oplus $ is commutative and associative, $\otimes $ is
associative, and $\otimes $ is distributive over $\oplus $.

\item[\textbf{(b)}] $\oplus $ has a neutral element $\varepsilon $, $\otimes 
$ has a neutral element $e$, and $\varepsilon $ is absorbing for $\otimes $,
namely, for any $a\in D$, $a\otimes \varepsilon =\varepsilon \otimes
a=\varepsilon $.
\end{enumerate}

We shall assume $\varepsilon \neq e$, and when convenient use $0_{D}$ for $%
\varepsilon $ and $1_{D}$ for $e$, write $+$ for $\oplus $, use the sum
symbol $\Sigma $ and omit the $\otimes $ symbol.

The triple $\left( D,\oplus ,\otimes \right) $ is a \emph{dioid} if in
addition to (a) and (b) the following condition (c) is satisfied:

\begin{enumerate}
\item[\textbf{(c)}] Existence of order relation condition: The binary
relation $\leq _{D}$ defined over $D$ via:%
\begin{equation*}
\forall a,b\in D,~(a\leq _{D}b\text{ }\Longleftrightarrow \exists c\in
D,~a\oplus c=b)\text{,}
\end{equation*}%
is an order relation. In this case we say that $D$ is canonically ordered
with respect to $\oplus $.
\end{enumerate}

Furthermore, we say that $\leq _{D}$ is complete if every subset of elements
of $D$ has a supremum in $D$, and that $\leq _{D}$ is dually complete if
every subset of elements of $D$ has an infimum in $D$. The dioid $\left(
D,\oplus ,\otimes \right) $ is complete (dually complete) if $\leq _{D}$ is
complete (dually complete) and the sum operation $\oplus $ can be extended
to arbitrary multisets of arguments such that $\otimes $ distributes over
the extended sums.
\end{definition}

\begin{remark}
\label{Rem_AntiSymmetryIsEnough}If $\left( D,\oplus ,\otimes \right) $ is a
semiring then the relation $\leq _{D}$ over $D$ is reflexive and transitive.
Thus, we can replace condition (c) by the condition that $\leq _{D}$ is
anti-symmetric. Note that some authors reserve the term dioid to idempotent
semirings, namely, to semirings $\left( D,\oplus ,\otimes \right) $ which
satisfy the additional condition $a\oplus a=a$ for every $a\in D$. One can
easily check that this condition implies that\ $\leq _{D}$ is an order
relation. We will refer to these dioids as idempotent dioids.
\end{remark}

\begin{remark}
\label{Rem_DioidsAndRingsDoNotIntersect}A ring $\left( D,\oplus ,\otimes
\right) $ is a semiring which satisfies the extra condition (c'): For any $%
a\in D$ there exists $b\in D$ such that $a\oplus b=\varepsilon $ (i.e., $%
\left( D,\oplus \right) $ is a commutative group). Thus, both rings \ and
dioids are special semirings. However, using $\varepsilon \neq e$ assumed in
Definition \ref{Def_Dioid}, (c) and (c') are mutually exclusive \cite[%
Section 1.3.4, Theorem 1]{GondranMinouxDioidBook2010}.
\end{remark}

The Boolean dioid $\mathbb{B}:=\left( \left\{ \varepsilon ,e\right\} ,\oplus
,\otimes \right) $ is the smallest dioid. It consists of just the two
neutral elements $\varepsilon $ and $e$ where $e\oplus e=e$. It is
idempotent and complete.

\subsection{Group Semirings}

Consider the following natural generalization of the familiar concept of a
group ring over a commutative ring.

\begin{definition}
\label{Def_GroupSemiring}Let $G$ be a group with identity element $1_{G}$
and let $K$ be a commutative semiring with identity $1_{K}\neq 0_{K}$.
Assume further that either $G$ is finite or, if $G$ is infinite then $K$ is
a complete dioid. We define the \emph{group semiring} $K\left[ G\right] $ of 
$G$ to be the set of all formal sums $\tsum\limits_{g\in G}a_{g}g$ where $%
a_{g}\in K$ for all $g\in G$. We define a pointwise addition over $K\left[ G%
\right] $ by 
\begin{equation*}
\tsum\limits_{g\in G}a_{g}g+\tsum\limits_{g\in G}b_{g}g:=\tsum\limits_{g\in
G}\left( a_{g}+b_{g}\right) g\text{,}
\end{equation*}%
and a convolution product over $K\left[ G\right] $ by%
\begin{equation*}
\left( \tsum\limits_{g\in G}a_{g}g\right) \left( \tsum\limits_{g\in
G}b_{g}g\right) :=\tsum\limits_{g\in G}\tsum\limits_{h\in G}\left(
a_{h}b_{h^{-1}g}\right) g=\tsum\limits_{g\in G}\tsum\limits_{h\in G}\left(
a_{gh^{-1}}b_{h}\right) g\text{.}
\end{equation*}
\end{definition}

\begin{proposition}
\label{Prop_k[G]IsSemiAndD}Let $G$ and $K$ be as in Definition \ref%
{Def_GroupSemiring}. Then $K\left[ G\right] $ is a semiring. Moreover, if $K$
is a dioid then $K\left[ G\right] $ is a dioid, and it is complete if $K$ is
complete.
\end{proposition}

\begin{proof}
It is routine to check that $K\left[ G\right] $ satisfies the semiring
axioms with respect to the two operations defined. The element $%
\tsum\limits_{g\in G}0_{K}g$ is the neutral element with respect to the
addition operation and the element $1_{K}1_{G}+\tsum\limits_{g\in
G\backslash \left\{ 1_{G}\right\} }0_{K}g$ is the neutral element with
respect to the multiplication operation. Of course, one can omit terms with $%
0_{K}$ and identify the subset $\left\{ s1_{G}|s\in K\right\} \subseteq K%
\left[ G\right] $ with $K$ and the subset $\left\{ 1_{K}g|g\in G\right\}
\subseteq K\left[ G\right] $ with $G$.

In order to prove that $K\left[ G\right] $ is a dioid whenever $K$ is, it
remains to prove the existence of a canonical order relation (condition (c)
of Definition \ref{Def_Dioid}). We abuse notation and write $+$ for the
addition operation over $K$ and over $K\left[ G\right] $. Define a binary
relation $\leq _{K\left[ G\right] }$ over $K\left[ G\right] $ by 
\begin{equation*}
\forall a,b\in K\left[ G\right] ,~(a\leq _{K\left[ G\right] }b\text{ }%
\Longleftrightarrow \exists c\in K\left[ G\right] ,~a+c=b)\text{.}
\end{equation*}%
By Remark \ref{Rem_AntiSymmetryIsEnough} it suffices to show that $\leq _{K%
\left[ G\right] }$ is anti-symmetric. Let:%
\begin{equation*}
a:=\tsum\limits_{g\in G}a_{g}g\text{, }b:=\tsum\limits_{g\in G}b_{g}g\text{, 
}c:=\tsum\limits_{g\in G}c_{g}g\text{, }d:=\tsum\limits_{g\in G}d_{g}g\text{.%
}
\end{equation*}%
We have to prove that if $a+c=b$ and $b+d=a$ then $a=b$. From $a+c=b$ we get 
$a_{g}+c_{g}=b_{g}$, $\forall g\in G$. Similarly, $b+d=a$ implies $%
b_{g}+d_{g}=a_{g}$, $\forall g\in G$. Since $K$ is a dioid, the first
equality implies $a_{g}\leq _{K}b_{g}$, $\forall g\in G$, and the second
equality implies $b_{g}\leq _{K}a_{g}$, $\forall g\in G$. Since $\leq _{K}$
is anti-symmetric, $a_{g}=b_{g}$, $\forall g\in G$ and this implies $a=b$.
We leave the verification of the completeness claim to the reader.
\end{proof}

\begin{notation}
Let $G$ be a group and $K$ be a commutative semiring which is a complete
dioid if $G$ is not finite. For any non-empty subset $S$ of $G$ define $%
\underline{S}:=$ $\tsum\limits_{g\in S}g$, and for any collection $\mathcal{S%
}$ of non-empty subsets of $G$ define $\underline{\mathcal{S}}:=\left\{ 
\underline{S}|S\in \mathcal{S}\right\} $.
\end{notation}

\begin{definition}
Let $G$ be a group and $K$ be a commutative semiring which is a complete
dioid if $G$ is not finite. Let $\mathcal{A}$ be a subsemiring of the group
semiring $K\left[ G\right] $. We shall say that $B\subseteq \mathcal{A}$ is
a \emph{basis} of $\mathcal{A}$ over $K$ if each element of $\mathcal{A}$
can be uniquely written as $\tsum\limits_{b\in B}a_{b}b$ where $a_{b}\in K$
for all $b\in B$. For example, $G$ is a basis of $K\left[ G\right] $ over $K$%
.
\end{definition}

\subsection{Schur Rings and S-partitions\label{Subsect_SRingsAndPartitions}}

In the special case of Definition \ref{Def_GroupSemiring}, where $K$ is a
ring and $G$ is a finite group, we obtain the usual group ring $K\left[ G%
\right] $. Schur rings, introduced by Schur in 1933 \cite{Schur1933}, are
special subrings of the group ring $K\left[ G\right] $, defined as follows.

\begin{definition}[{\hspace{0.02in}\negthinspace \noindent \protect\cite[%
Section 2, Definition 2.1]{MuzychukPonomarenko2009}}]
\label{Def_SchurRing}Let $K$ be a ring and let $G$ be a finite group. A
subring $\mathcal{A}$ of $K\left[ G\right] $ is a \emph{Schur ring} over $G$
if there exists a partition $\mathcal{S}$ of $G$ which has the following
properties:

a. The set $\underline{\mathcal{S}}$ is a basis of $\mathcal{A}$ over $K$.

b. $\left\{ 1_{G}\right\} \in \mathcal{S}$

c. $X^{-1}\in \mathcal{S}$ for all $X\in \mathcal{S}$.

We shall call a partition $\mathcal{S}$ which has the properties (a)-(c), an 
\emph{S-partition}.
\end{definition}

Note that in the original Definition 2.1 of \cite{MuzychukPonomarenko2009} $%
K=\mathbb{C}$.

Let $\mathcal{A}$ be a Schur ring over a finite group $G$, with an
associated S-partition $\mathcal{S}=\left\{ S_{1},...,S_{h}\right\} $
(Definition \ref{Def_SchurRing} with $K=\mathbb{C}$ for definiteness). For
any $1\leq i,j\leq h$ the product $\underline{S_{i}}\underline{S_{j}}$ can
be expressed as a linear combination of the elements of $\underline{\mathcal{%
S}}$, as the latter is a basis of $\mathcal{A}$. It follows, from the group
multiplication law, that the coefficients of the basis elements are
non-negative integers. Writing%
\begin{equation*}
\underline{S_{i}}\underline{S_{j}}=\tsum\limits_{k=1}^{h}s_{ij}^{k}%
\underline{S_{k}}\text{,}
\end{equation*}%
we refer to the $h^{3}$ coefficients $s_{ij}^{k}$\ as the\ \emph{structure
constants} of the Schur ring. Note that since $\mathbb{Z}\subseteq \mathbb{C}%
\subseteq \mathcal{A}$, we can think of the structure constants as elements
of the Schur ring.

\begin{remark}
\label{Rem_s-ringDefWithoutC}One can reformulate the definition of an
S-partition in terms of a key property of the structure constants $%
s_{ij}^{k} $ by replacing condition (a) of Definition \ref{Def_SchurRing}
with (a'): For any $1\leq i,j,k\leq h$ and any $z\in S_{k}$, the number of
distinct pairs $\left( x,y\right) \in S_{i}\times S_{j}$ such that $xy=z$
depends only on $\left( i,j,k\right) $ and is independent of the specific
choice of $z$. Note that this number equals $s_{ij}^{k}$, and that this
version of the definition does not involve the explicit introduction of a
group ring.
\end{remark}

\subsection{Finite Cyclic Groups}

Let $n$ be a positive integer. We will work with the realization of a cyclic
group of order $n$ as $G:=\left( \mathbb{Z}_{n},+_{n}\right) $, where $%
\mathbb{Z}_{n}:=\left\{ 0,1,...,n-1\right\} $ and $+_{n}$ denotes addition
modulo $n$. Frequently we write $+$ for $+_{n}$. Recall that $Aut\left(
G\right) =\left( \mathbb{Z}_{n}^{\ast },\cdot _{n}\right) $, where $\mathbb{Z%
}_{n}^{\ast }=\left\{ a\in \mathbb{Z}_{n}|\gcd \left( a,n\right) =1\right\} $%
, and $\cdot _{n}$ denotes multiplication modulo $n$. Using this
presentation, any $A\leq Aut\left( G\right) $ acts on $G$ by $\cdot _{n}$
multiplication.

For any prime $p$ and any multiplicative subgroup $A$ of $\left( \mathbb{Z}%
_{p}^{\ast },\cdot _{p}\right) $ we set $\Pi \left( A\right) :=\left\{
\left\{ 0\right\} \right\} \cup \left\{ Ax|x\in \mathbb{Z}_{p}^{\ast
}\right\} $ - the partition of the additive group $G:=\left( \mathbb{Z}%
_{p},+_{p}\right) $ whose non-trivial parts are all the cosets of the
multiplicative subgroup $A\leq \left( \mathbb{Z}_{p}^{\ast },\cdot
_{p}\right) $. Note that $\Pi \left( A\right) $ is not a coset partition of $%
G$ in the sense of the Example \ref{Example_Classical}. We prove that $\Pi
\left( A\right) $ is an S-partition of $G$. By Example \ref%
{Example_ZnMultiplicative}, $\Pi \left( A\right) $ is a d-partition of $G$,
and hence it remains to verify property (a') of Remark \ref%
{Rem_s-ringDefWithoutC}. Let $k$ be the number of distinct cosets of $A$ in $%
\mathbb{Z}_{p}^{\ast }$. Let $i,j,l\in \left\{ 1,...,k\right\} $ and let $%
h,g\in Ax_{l}$. Note that there exists a unique $a\in A$ such that $h=ag$.
Suppose that there are exactly $m_{g}$ distinct pairs $\left( u,v\right) \in
Ax_{i}\times Ax_{j}$ such that $u+v=g$. Then, the $m_{g}$ distinct pairs $%
\left( au,av\right) \in Ax_{i}\times Ax_{j}$ satisfy $au+av=a\left(
u+v\right) =ag=h$. It follows that $m_{g}\leq m_{h}$ and by symmetry $%
m_{g}=m_{h}$.

Gordon gave in \cite{BGordonPacific1964} an elegant proof that the
partitions $\Pi \left( A\right) $ exhaust all of the S-partitions of $\left( 
\mathbb{Z}_{p},+_{p}\right) $. Thus we have a complete classification of
S-partitions of cyclic groups of prime order (see \cite%
{MuzychukPonomarenko2009} for generalizations):

\begin{theorem}[\protect\cite{BGordonPacific1964}]
\label{Th_BGordon}Let $p$ be a prime. If $A$ is a multiplicative subgroup of 
$\left( \mathbb{Z}_{p}^{\ast },\cdot _{p}\right) $, then $\Pi \left(
A\right) $ is an S-partition of $\left( \mathbb{Z}_{p},+_{p}\right) $.
Conversely, if $\Pi $ is an S-partition of $\left( \mathbb{Z}%
_{p},+_{p}\right) $ then $\Pi =$ $\Pi \left( A\right) $ for some $A\leq 
\mathbb{Z}_{p}^{\ast }$.
\end{theorem}

The following lemma provides a useful sufficient condition for a d-partition
of a prime order group to be an S-partition.

\begin{lemma}
\label{Lem_ZpSingletonOrDoubleImplyMult}Let $p$ be a prime and let $%
G:=\left( \mathbb{Z}_{p},+_{p}\right) $. Let $\Pi $ be a d-partition of $G$,
with $e_{\Pi }=\left\{ 0\right\} $. The following two statements hold:

1. If there exists $\pi \in \Pi \backslash \left\{ e_{\Pi }\right\} $ such
that $\left\vert \pi \right\vert =1$ then $\Pi =\Pi \left( \left\{ 1\right\}
\right) $ (the singleton partition of $G$).

2. If for all $\pi \in \Pi \backslash \left\{ e_{\Pi }\right\} $ we have $%
\left\vert \pi \right\vert \geq 2$, and there exists $\pi \in \Pi $ such
that $\left\vert \pi \right\vert =2$, then $\Pi =\Pi \left( \left\{
-1,1\right\} \right) $.
\end{lemma}

\begin{proof}
1. Assume that $\pi \in \Pi \backslash \left\{ e_{\Pi }\right\} $ satisfies $%
\left\vert \pi \right\vert =1$. Then $\pi =\left\{ x\right\} $ for some $%
x\in \mathbb{Z}_{p}^{\ast }$, and we have, for any positive integer $k$, 
\begin{equation*}
\underset{k\text{ summands}}{\underbrace{\pi +\cdots +\pi }}=\left\{ k\cdot
_{p}x\right\} \text{.}
\end{equation*}%
Since $x$ generates $G$ the claim follows.

2. Let $\pi \in \Pi $ be such that $\left\vert \pi \right\vert =2$. Then $%
\pi =\left\{ x_{1},x_{2}\right\} $ with $x_{1}\neq x_{2}\in \mathbb{Z}%
_{p}^{\ast }$. Since $\Pi $ satisfies the inverse property, either $\pi
=-\pi $, in which case $x_{2}=-x_{1}$, or $-\pi \in \Pi $ is disjoint from $%
\pi $. In this second case, we get $\pi +\left( -\pi \right) =e_{\Pi }\cup
\left\{ -\left( x_{2}-x_{1}\right) ,x_{2}-x_{1}\right\} $, and by our
assumptions we must have $\left\{ -\left( x_{2}-x_{1}\right)
,x_{2}-x_{1}\right\} \in \Pi $. Hence we can assume, without loss of
generality, that $\pi =\left\{ -x,x\right\} $ for some $x\in \mathbb{Z}%
_{p}^{\ast }$. Now set $\pi _{k}:=k\cdot _{p}\pi $ for $k\in \mathbb{Z}%
_{p}^{\ast }$, and $\pi _{0}:=e_{\Pi }$. Observe that $\pi _{1}=\pi $ and
that for any $k>1$, we have $\pi +\pi _{k-1}=\pi _{k-2}\cup \pi _{k}$. This
yields, by an easy induction, that every non-identity part of $\Pi $ is of
the form $\left\{ -\left( k\cdot _{p}x\right) ,k\cdot _{p}x\right\} $ which
gives our claim.
\end{proof}

We need the following standard lemma, whose easy proof is omitted, for the
ensuing corollary.

\begin{lemma}
\label{Lem_NegatingA}Let $p$ be a prime and let $A\leq \mathbb{Z}_{p}^{\ast
} $. Then:

\begin{enumerate}
\item[a.] Either $A=-A$, or $-A$ is some non-trivial multiplicative coset of 
$A$.

\item[b.] The following conditions are equivalent:

\begin{enumerate}
\item[(b1)] $A=-A$.

\item[(b2)] $-1\in A$.

\item[(b3)] $\left\vert A\right\vert $ is even.
\end{enumerate}
\end{enumerate}
\end{lemma}

\begin{corollary}
\label{Coro_MultiplicativePartitionsUnderInverse}Let $p$ be a prime and let $%
A\leq \mathbb{Z}_{p}^{\ast }$. If $A=-A$ then $\pi =-\pi $ for all $\pi \in
\Pi \left( A\right) $, and if $A\neq -A$ then for every$\ \pi \in \Pi \left(
A\right) \backslash \left\{ \left\{ 0\right\} \right\} $ there exists $\pi
^{\prime }\neq \pi $ such that $\pi =-\pi ^{\prime }$ . In particular, if $%
\Pi \left( A\right) =\left\{ \pi _{0},\pi _{1},\pi _{2}\right\} $ then $%
p\equiv 1\left( \func{mod}4\right) $ implies $\pi =-\pi $ for all $\pi \in
\Pi \left( A\right) $ and $p\equiv 3\left( \func{mod}4\right) $ implies $\pi
_{1}=-\pi _{2}$.
\end{corollary}

\subsection{Additive Combinatorics\label{Subsect_AdditiveCombResults}}

Following the terminology of \cite{TaoVuAdditive2007}, an additive set is a
non-empty subset $S$ of some abelian group $\left( G,+\right) $ written
additively. Given two additive sets $A$ and $B$ with the same underlying
group, define $A+B:=\left\{ a+b|a\in A,b\in B\right\} $, and $-A:=\left\{
-a|a\in A\right\} $. We shall say that $A$ is symmetric if $A=-A$. The set
difference of two sets $A$ and $B$ will be denoted $A\backslash B$ and
should not be confused with the arithmetic difference $A-B:=A+\left(
-B\right) $. If $A$ is a subset of $G$ and $G$ is clear from context, we
write $\overline{A}$ for $G\backslash A$. For any integers $a\leq b$, the
(discrete) closed interval $\left[ a,b\right] $ is the integer subset $%
\left\{ a\leq x\leq b|x\in \mathbb{Z}\right\} $, and we use this notation
also for $a,b\in \mathbb{Z}_{n}$ when $0\leq a\leq b\leq n-1$. An additive
subset $S$ of an additive group $\left( G,+\right) $ is called an arithmetic
progression with step $\delta $ (or, in short, a $\delta $-progression) if
there exist $a_{0},\delta \in G$ and $N\in \mathbb{N}_{0}$ such that $%
S=\left\{ a_{0}+j\cdot \delta |j\in \mathbb{N}_{0},~0\leq j\leq N\right\}
=a_{0}+\left[ 0,N\right] \cdot \delta $. Note that we have the easy addition
formula for two $\delta $-progressions:%
\begin{equation*}
\left( a_{0}+\left[ 0,N_{1}\right] \cdot \delta \right) +\left( b_{0}+\left[
0,N_{2}\right] \cdot \delta \right) =\left( a_{0}+b_{0}+\left[ 0,N_{1}+N_{2}%
\right] \cdot \delta \right) \text{.}
\end{equation*}

\begin{theorem}[Cauchy-Davenport Inequality, \protect\cite{Cauchy1813},%
\protect\cite{Davenport1935},\protect\cite{TaoVuAdditive2007}]
\label{Th_CauchyDavenport}Let $p$ be a prime and let $G=\left( \mathbb{Z}%
_{p},+_{p}\right) $. For any two additive subsets $A,B\subseteq G$ we have%
\begin{equation*}
\left\vert A+B\right\vert \geq \min \left( \left\vert A\right\vert
+\left\vert B\right\vert -1,p\right) \text{.}
\end{equation*}
\end{theorem}

\begin{theorem}[Vosper's Theorem, \protect\cite{Vosper1956},\protect\cite%
{TaoVuAdditive2007}]
\label{Th_Vosper}Let $p$ be a prime and let $A$ and $B$ be additive subsets
of $\mathbb{Z}_{p}$ such that $\left\vert A\right\vert ,\left\vert
B\right\vert \geq 2$ and $\left\vert A+B\right\vert \leq p-2$. Then $%
\left\vert A+B\right\vert =\left\vert A\right\vert +\left\vert B\right\vert
-1$ if and only if $A$ and $B$ are arithmetic progressions with the same
step.
\end{theorem}

\section{D-partitions of Groups \label{Sect_D-partitionsOfGroups}}

In this section we prove several general properties of d-partitions. We
begin with a proof of Theorem \ref{Th_DioidOutOfd-partition} which claims
that if $\Pi $ is a d-partition of $G$, then the set $D_{\Pi }$ of all
possible unions of parts of $\Pi $ is a dioid where $\oplus $ and $\otimes $
are taken to be, respectively, set union and setwise product.

\begin{proof}[Proof of Theorem \protect\ref{Th_DioidOutOfd-partition}]
It is straightforward to check that the conditions in Definition \ref%
{Def_Dioid} are satisfied. In particular, the empty set $\emptyset $ is a
neutral element with respect to $\oplus $, we have $\emptyset \otimes \pi
=\pi \otimes \emptyset =\emptyset $ for any $\pi \in \Pi $, and $e_{\Pi }$
is a neutral element with respect to $\otimes $. Since set union is an
idempotent binary operation, $\left( D_{\Pi },\oplus ,\otimes \right) $ is
canonically ordered with respect to $\oplus $ (see Remark \ref%
{Rem_AntiSymmetryIsEnough}), where the canonical order relation $\leq
_{D_{\Pi }}$(see Definition \ref{Def_Dioid}(c)) is set inclusion.
\end{proof}

\begin{remark}
Keeping the notation and assumptions of Theorem \ref%
{Th_DioidOutOfd-partition}, we claim that $D_{\Pi }$ is a complete and
dually complete dioid. To see this observe that the supremum of a subset of $%
D_{\Pi }$ is the union of its elements, and that $G=\tbigcup\limits_{\pi \in
\Pi }\pi $ is the unique maximal element with respect to $\leq _{D_{\Pi }}$ (%
$G$ is absorbing with respect to $\otimes $ and all the elements of $D_{\Pi
} $ except for $\emptyset $). The infimum of a subset of $D_{\Pi }$ is the
intersection of all of its elements. To see this use the fact that each
element of $D_{\Pi }$ is a union of some elements of $\Pi $, and the
intersection of any two distinct elements of $\Pi $ is empty. Finally,
unions of infinite families of subsets of $G$ are well defined and the
setwise product distributes over such unions. Also note that since $\Pi $ is
a partition, $D_{\Pi }$ is in bijection with the power set of $\Pi $.
\end{remark}

\subsection{Basic Properties of d-partitions\label{Subsect_BasicPropOfD}}

The following lemma summarizes basic properties of the identity element of a
d-partition.

\begin{lemma}
\label{Lem_DioidIdentityIsSubgroup}Let $G$ be a group and let $\Pi $ be a
partition of $G$ satisfying the closure and existence of an identity element
properties (Definition \ref{Def_d-partition}, (a) and (b)). Then:

a. $e$ is unique.

b. $e$ is a subgroup of $G$.

c. Each $\pi \in \Pi $ is a union of double cosets of $e$.
\end{lemma}

\begin{proof}
a. Standard.

b. Since $e$ is a non-empty subset of $G$, and $e^{2}=e$, we get that $e$ is
closed under group multiplication. Let $\pi \in \Pi $ be the unique part
containing $1_{G}$. Then, $e\subseteq e\pi =\pi $ whence, $e=\pi $ and $%
1_{G}\in e$. Let $g\in e$ and let $\sigma $ be the unique part containing $%
g^{-1}$. Then, $1_{G}=gg^{-1}\in e\sigma =\sigma $ whence $e=\sigma $. It
follows that $e$ is a subgroup of $G$.

c. For any $\pi \in \Pi $ we have 
\begin{equation*}
\pi =e\pi e=\tbigcup\limits_{x\in \pi }exe\text{,}
\end{equation*}%
and therefore $\pi $ is a union of double cosets of $e$.
\end{proof}

\begin{remark}
\label{Rem_T2}Lemma \ref{Lem_DioidIdentityIsSubgroup} is the same as \cite[%
Lemma 1.2]{Tamaschke1968}.
\end{remark}

The following concept can be used to give an alternative characterization of
d-partitions.

\begin{definition}
Let $G$ be a group and let $\Pi $ be a partition of $G$. We say that $\Pi $
satisfies the \emph{intersection property} if for all $\pi _{1},\pi _{2},\pi
\in \Pi $ such that $\pi \subseteq \pi _{1}\pi _{2}$, for all $x\in \pi _{1}$
and for all $y\in \pi _{2}$ it holds that 
\begin{equation*}
\left( x\pi _{2}\right) \cap \pi \neq \emptyset ,~~\left( \pi _{1}y\right)
\cap \pi \neq \emptyset \text{.}
\end{equation*}
\end{definition}

\begin{theorem}
Let $\Pi $ be a partition of the group $G$ satisfying the closure and
existence of an identity element properties (Definition \ref{Def_d-partition}%
, (a) and (b)). Then $\Pi $ satisfies the intersection property if and only
if $\Pi $ satisfies the inverse property (Definition \ref{Def_d-partition},
(c)), equivalently, if and only if $\Pi $ is a d-partition.
\end{theorem}

\begin{proof}
1. We suppose that $\Pi $ is a d-partition, and prove that it satisfies the
intersection property. Let $\pi _{1},\pi _{2},\pi \in \Pi $ be such that $%
\pi \subseteq \pi _{1}\pi _{2}$. Since $\pi \subseteq \pi _{1}\pi _{2}$,
there exist $\gamma \in \pi $, $\alpha \in \pi _{1}$, $\beta \in \pi _{2}$
such that $\gamma =\alpha \beta $. It follows that $\alpha =\gamma \beta
^{-1}$. Since $\beta ^{-1}\in \pi _{2}^{-1}$, this shows that $\pi _{1}\cap
\pi \pi _{2}^{-1}\neq \emptyset $. By (c) of Definition \ref{Def_d-partition}%
, $\pi _{2}^{-1}\in \Pi $. Hence, $\pi _{1}\cap \pi \pi _{2}^{-1}\neq
\emptyset $ implies $\pi _{1}\subseteq \pi \pi _{2}^{-1}$. Hence, for every $%
x\in \pi _{1}$ there exist $u\in \pi _{2}$ and $z\in \pi $ such that $%
x=zu^{-1}$ which is equivalent to $xu=z$. Thus $z\in \left( x\pi _{2}\right)
\cap \pi $ and $\left( x\pi _{2}\right) \cap \pi \neq \emptyset $. The proof
that $\left( \pi _{1}y\right) \cap \pi \neq \emptyset $ for all $y\in \pi
_{2}$ is similar.

2. We suppose that $\Pi $ satisfies conditions (a) and (b) of Definition \ref%
{Def_d-partition} and the intersection property, and prove that it satisfies
the inverse property. The proof borrows from the ideas of the proof of \cite[%
Lemma 2]{BGordonPacific1964}. Let $\pi \in \Pi $ be arbitrary. We wish to
prove that $\pi ^{-1}\in \Pi $. Fix $y\in \pi $ and let $C$ be the unique
part in $\Pi $ to which $y^{-1}$ belongs. Then $1_{G}\in yC$, implying $%
e\subseteq e\pi C=\pi C$. Let $x\in \pi $ be arbitrary. By the intersection
property we get that $xC\cap e\neq \emptyset $. This implies $C\cap
x^{-1}e\neq \emptyset $. Equivalently, by taking inverses and using Lemma %
\ref{Lem_DioidIdentityIsSubgroup}(b), $C^{-1}\cap ex\neq \emptyset $. This
implies $x\in eC^{-1}$. Now $C\in \Pi $\ implies that $eC=Ce=C$. Taking
inverses and using Lemma \ref{Lem_DioidIdentityIsSubgroup}(b) this gives $%
eC^{-1}=C^{-1}$. Hence $x\in eC^{-1}$ implies $x\in C^{-1}$. We have proved
that $\pi \subseteq C^{-1}$. Interchanging the roles of $y$ and $y^{-1}$ and
of $\pi $ and $C$, the above reasoning implies that $C\subseteq \pi ^{-1}$.
Taking inverses in the previous relation $\pi \subseteq C^{-1}$ we get $\pi
^{-1}\subseteq C$. It follows that $C=\pi ^{-1}$ whence $\pi ^{-1}\in \Pi $.
\end{proof}

The analogy between Schur rings (see Section \ref%
{Subsect_SRingsAndPartitions}) and Schur dioids suggests the introduction of
structure constants for Schur dioids. Their existence follows from the
following lemma, whose proof is immediate from Definition \ref%
{Def_d-partition}.

\begin{lemma}
\label{Lem_Def_DioidStructureConst}Let $G$ be a group, $\Pi $ a d-partition
of $G$, and $\left( D_{\Pi },\oplus ,\otimes \right) $ the associated Schur
dioid over $G$. Then there exists a function $d:\Pi ^{3}\rightarrow \left\{
\emptyset ,e_{\Pi }\right\} $ whose values $d_{\sigma ,\pi }^{\tau }\in
\left\{ \emptyset ,e_{\Pi }\right\} $, where $\pi ,\sigma ,\tau \in \Pi $
are arbitrary, satisfy 
\begin{equation*}
\pi \otimes \sigma =\underset{\tau \in \Pi }{\oplus }\left( d_{\sigma ,\pi
}^{\tau }\otimes \tau \right) \text{,}
\end{equation*}%
for all $\pi ,\sigma \in \Pi $. The $d_{\sigma ,\pi }^{\tau }$ will be
called the structure constants of $\left( D_{\Pi },\oplus ,\otimes \right) $.
\end{lemma}

\begin{remark}
\label{Rem_01}For notational convenience we will write $0$ and $1$ for the
values of the dioid structure constants, instead of, respectively, $%
\emptyset $ and $e_{\Pi }$.
\end{remark}

Now we can define isomorphic d-partitions.

\begin{definition}
\label{Def_IsomorphicDPartitions}Let $\Pi _{1}$ and $\Pi _{2}$ be two
d-partitions with structure constants $\left( d_{1}\right) _{\sigma _{1},\pi
_{1}}^{\tau _{1}}$, $\tau _{1},\pi _{1},\sigma _{1}\in \Pi _{1}$ and $\left(
d_{2}\right) _{\sigma _{2},\pi _{2}}^{\tau _{2}}$, $\tau _{2},\pi
_{2},\sigma _{2}\in \Pi _{2}$. Then $\Pi _{1}$ and $\Pi _{2}$ are isomorphic
if there exists a bijection $f:\Pi _{1}\rightarrow \Pi _{2}$ such that for
any $\tau _{1},\pi _{1},\sigma _{1}\in \Pi _{1}$ we have $\left(
d_{1}\right) _{\sigma _{1},\pi _{1}}^{\tau _{1}}=\left( d_{2}\right)
_{f\left( \sigma _{1}\right) ,f\left( \pi _{1}\right) }^{f\left( \tau
_{1}\right) }$.
\end{definition}

\begin{remark}
\label{Rem_T3}Tamaschke defines a homomorphism from one S-semigroup into
another (see \cite[Definition 2.1]{Tamaschke1968}) which extends in a
natural way to the associated dioid structures. For bijective homomorphisms,
this definition coincides with Definition \ref{Def_IsomorphicDPartitions}
above.
\end{remark}

\begin{remark}
\label{Rem_DPartitionUnderIso}If $\Pi $ is a d-partition of a group $G$ and $%
\iota :G\rightarrow H$ is a group isomorphism, then $\iota \left( \Pi
\right) :=\left\{ \iota \left( \pi \right) |\pi \in \Pi \right\} $ is a
d-partition of $H$ which is isomorphic to $\Pi $. However, non-isomorphic
groups might still have isomorphic d-partitions. See for instance Corollary %
\ref{Coro_Supplement_construction}.
\end{remark}

\subsection{S-partitions vs. d-partitions\label{Subsect_SAndDPartitions}}

Using the concept of a group semiring (Definition \ref{Def_GroupSemiring}
and Proposition \ref{Prop_k[G]IsSemiAndD}), we present in this section
another approach for constructing a dioid from a suitable group partition,
which is on the same footing as the construction of a Schur ring from an
S-partition.

\begin{definition}
\label{Def_SchurDioid}Let $G$ be a group and $K$ be a commutative dioid
which is a complete dioid if $G$ is not finite. A subdioid $\mathcal{A}$ of $%
K\left[ G\right] $ is called a generalized 1-Schur dioid over $G$ if there
exists a partition $\Pi $ of $G$ which has the following properties:

\begin{enumerate}
\item[a.] The set $\underline{\Pi }$ is a basis of $\mathcal{A}$ over $K$.

\item[b.] $\left\{ 1_{G}\right\} \in \Pi $.

\item[c.] $X^{-1}\in \Pi $ for all $X\in \Pi $.
\end{enumerate}

We shall call a partition $\Pi $ which has the properties (a)-(c), a
generalized 1d-partition.
\end{definition}

The similarity between Definitions \ref{Def_SchurDioid} and \ref%
{Def_SchurRing} is evident. Indeed, it can be verified that 1d-partitions by
Definition \ref{Def_d-partition} arise as a special case of Definition \ref%
{Def_SchurDioid}, as stated in the following lemma.

\begin{lemma}
Definition \ref{Def_d-partition} with $e_{\Pi }=\left\{ 1_{G}\right\} $ and
Definition \ref{Def_SchurDioid} with $K=\mathbb{B}$ (the Boolean dioid)
coincide.
\end{lemma}

For the rest of the paper we work with Definition \ref{Def_d-partition}. It
would be interesting to study dioids which arise from other choices of $K$.
Here we just comment on one possible direction. As Schur rings are mainly
studied over fields, it makes sense to look for an analogue structure in the
case of commutative dioids. An obvious idea is to look at commutative dioids 
$K$ for which every element besides $\varepsilon $ is invertible with
respect to $\otimes $. We call such dioids \emph{d-fields} (see also \cite[%
Definition 1.5.2.3]{GondranMinouxDioidBook2010}). For example, $\mathbb{B}$
is a d-field. However, we are unaware of a systematic study of d-fields. We
offer the following observation whose proof is given in the appendix
(Section \ref{Appendix_ProofOdDfield}).

\begin{proposition}
\label{Prop_d-field}Let $\left( D,\oplus ,\otimes \right) $ be a d-field.
Then either, for each $d\in D\backslash \left\{ \varepsilon \right\} $ all
finite sums of the form $d,d\oplus d,d\oplus d\oplus d,...$ are distinct, or 
$D$ is an idempotent dioid. Furthermore, if $D$ is an idempotent dioid then
either $D=\mathbb{B}=\left( \left\{ \varepsilon ,e\right\} ,\oplus ,\otimes
\right) $ or $D$ has no largest element, and, in particular, $D$ is infinite
and is not complete.
\end{proposition}

Our last result in this section is the following direct connection between
S-partitions and d-partitions.

\begin{proposition}
\label{Prop_s_Implies_d}Every S-partition $\mathcal{S}$ of a finite group $G$
over a commutative ring $R$ of characteristic zero is a 1d-partition.
Furthermore, choosing some numbering of the parts of $\mathcal{S}$, and
denoting the structure constants of the associated Schur ring by $s_{ij}^{k}$
and those of the associated Schur dioid by $d_{ij}^{k}$, we have that $%
d_{ij}^{k}=0$ if and only if $s_{ij}^{k}=0$ and $d_{ij}^{k}=1$ if and only
if $s_{ij}^{k}>0$.
\end{proposition}

\begin{proof}
Let $\mathcal{S}$ be an S-partition of a finite group $G$ over $R$ and let $%
\mathcal{A}$ be the associated Schur ring. The existence of an identity, and
the inverse property (Definition \ref{Def_d-partition}, (b) and (c)) follow
immediately from (b) and (c) of Definition \ref{Def_SchurRing}. For the
closure property (Definition \ref{Def_d-partition} (a)), let $S_{1}$ and $%
S_{2}$ be any two elements of $\mathcal{S}$. We have to prove that $%
S_{1}S_{2}$ is a union of elements of $\mathcal{S}$. Since $\underline{%
\mathcal{S}}$ is a basis of $\mathcal{A}$ over $R$, and $\underline{S_{1}}~%
\underline{S_{2}}\in \mathcal{A}$, we have $\underline{S_{1}}~\underline{%
S_{2}}=\tsum\limits_{S\in \mathcal{S}}$ $b_{S}\underline{S}$ where $b_{S}\in
R$ are uniquely determined for all $S\in \mathcal{S}$. Using the definition
of the product of the formal sums $\underline{S_{1}}$ and $\underline{S_{2}}$%
, and the assumption that $R$ is of characteristic zero, we get that $%
b_{S}=n_{S}1_{R}$, where, for any $s\in S$, the non-negative integer $n_{S}$
counts the number of distinct pairs $\left( s_{1},s_{2}\right) $ with $%
s_{1}\in S_{1}$ and $s_{2}\in S_{2}$ such that $s_{1}s_{2}=s$. Note that $%
n_{S}$ does not depend on the specific choice of $s\in S$ (see Remark \ref%
{Rem_s-ringDefWithoutC}). We show that if $\left( S_{1}S_{2}\right) \cap
S_{3}\neq \emptyset $ then $S_{3}\subseteq S_{1}S_{2}$. Suppose that $x\in
S_{1}$ and $y\in S_{2}$ are such that $xy\in S_{3}$. Then $n_{S_{3}}\geq 1$,
and hence, for every $s\in S_{3}$ there exists a pair $\left(
s_{1},s_{2}\right) $ with $s_{1}\in S_{1}$ and $s_{2}\in S_{2}$ such that $%
s_{1}s_{2}=s$, proving $S_{3}\subseteq S_{1}S_{2}$. The fact that $\left(
S_{1}S_{2}\right) \cap S_{3}\neq \emptyset $ implies $S_{3}\subseteq
S_{1}S_{2}$ proves that $S_{1}S_{2}$ is a union of elements of $\mathcal{S}$%
. The relation between the structure constants follows easily.
\end{proof}

\section{Inducing d-partitions from given ones \label{Sect_Inducing}}

In this section we prove several results which share the following feature:
A d-partition $\Pi $ of a group $G$ induces a d-partition $\Pi ^{\prime }$
of a related group. We start with the proof of Theorem \ref%
{Th_General_1d_connection}.

\begin{proof}[Proof of Theorem \protect\ref{Th_General_1d_connection}]
Observe that any $\pi \in \Pi _{<A}$ is contained in $A$ while any $\pi \in
\Pi _{>A}$ is a union of non-trivial double cosets of $A$. Hence $\Pi
_{<A}\cap \Pi _{>A}=\emptyset $.

a. Suppose that $\left\{ \Pi _{<A},\Pi _{>A}\right\} $ is a partition of $%
\Pi $. We have to prove that $\Pi ^{\prime }=\Pi _{>A}\cup \left\{ A\right\} 
$ is a d-partition of $G$ with $e_{\Pi ^{\prime }}=A$. We have%
\begin{equation}
\left( \tbigcup\limits_{\pi \in \Pi _{<A}}\pi \right) \cup \left(
\tbigcup\limits_{\pi \in \Pi _{>A}}\pi \right) =\tbigcup\limits_{\pi \in \Pi
}\pi =G\text{.}  \label{Eq_U<U>=G}
\end{equation}%
Since $\tbigcup\limits_{\pi \in \Pi _{<A}}\pi \subseteq A$ it follows that
the set $\tbigcup\limits_{\pi \in \Pi _{>A}}\pi $\ contains each non-trivial
double coset of $A$. Since each part in $\Pi _{>A}$ is a union of
non-trivial double cosets of $A$, we can conclude that $\Pi _{>A}$ is a
partition of $G\backslash A$ whose parts are unions of double cosets of $A$.
Hence, by Equation (\ref{Eq_U<U>=G}), $\tbigcup\limits_{\pi \in \Pi
_{<A}}\pi =A$. The claim that $\Pi ^{\prime }$ is a d-partition with $e_{\Pi
^{\prime }}=A$ follows easily.

Suppose that $\left\{ \Pi _{<A},\Pi _{>A}\right\} $ is not a partition of $%
\Pi $. We prove that $\Pi ^{\prime }:=\Pi _{>A}\cup \left\{ A\right\} $ is
not a partition of $G$. Since $\Pi _{<A}\subseteq \Pi $, $\Pi _{>A}\subseteq
\Pi $ and $\Pi _{<A}\cap \Pi _{>A}=\emptyset $, there exists some $\pi \in
\Pi $ such that $\pi \notin $ $\Pi _{<A}\cup \Pi _{>A}$. This implies $\pi
\cap A\subset \pi $, and $\pi \cap \tbigcup\limits_{\sigma \in \Pi
_{>A}}\sigma =\emptyset $. Consequently, $\pi \cap \tbigcup\limits_{\sigma
\in \Pi ^{\prime }}\sigma =\pi \cap A\subset \pi $, proving that $\Pi
^{\prime }$ is not a partition of $G$.

b. Clearly $\Pi $ is a partition of $G$. We prove that it has the closure
property. Let $\pi _{1},\pi _{2}\in \Pi $. If $\pi _{1},\pi _{2}\in \Pi _{A}$
then $\pi _{1}\pi _{2}$ is a union of parts from $\Pi _{A}\subseteq \Pi $.
If $\pi _{1},\pi _{2}\in \Pi ^{\prime }\backslash \left\{ A\right\} $ then $%
\pi _{1}\pi _{2}$ is a union of parts from $\Pi ^{\prime }\backslash \left\{
A\right\} $ and possibly also $A$ (which is the union of all parts in $\Pi
_{A}$). If $\pi _{1}\in \Pi _{A}$ and $\pi _{2}\in \Pi ^{\prime }\backslash
\left\{ A\right\} $, then $\pi _{1}\subseteq A$ and $\pi _{2}$ is a union of
double cosets of $A$, and so $\pi _{1}\pi _{2}=\pi _{2}$. A similar argument
applies if $\pi _{1}\in \Pi ^{\prime }\backslash \left\{ A\right\} $ and $%
\pi _{2}\in \Pi _{A}$. The claim $e_{\Pi }=e_{\Pi _{A}}$ and the inverse
property are easy to verify. Finally, taking $\Pi _{A}$ to be the partition
of $A$ whose parts are the conjugacy classes of $A$ proves the last
assertion of (b).

c. We prove that $\Pi _{<A}$ is a d-partition of $A$. By assumption $\Pi
_{<A}$ is a partition of $A$. For any $\pi _{1},\pi _{2}\in \Pi _{<A}$ we
have that $\pi _{1}\pi _{2}\subseteq A$ because of the closure of $A$ under
multiplication. On the other hand, since $\Pi $ is a d-partition of $G$, $%
\pi _{1}\pi _{2}$ is a union of parts of $\Pi $. Hence $\pi _{1}\pi _{2}$ is
a union of parts of $\Pi _{<A}$. Moreover, $1_{G}\in A$ implies $e_{\Pi
}\subseteq A$, and since $A^{-1}=A$ we get that $\Pi _{<A}$ has the inverse
property.

In order to prove that $\Pi ^{\prime }$ is a d-partition of $G$, we first
prove that it is a partition of $G$. It is easily seen that the elements of $%
\Pi ^{\prime }$ are non-empty and that their union is equal to $G$. We prove
that distinct elements of $\Pi ^{\prime }$ are disjoint. For this it
suffices to show that if $\pi _{1},\pi _{2}\in \Pi \backslash \Pi _{<A}$,
and $\left( A\pi _{1}A\right) \cap \left( A\pi _{2}A\right) \neq \emptyset $%
, then $A\pi _{1}A=A\pi _{2}A$. Assuming $\left( A\pi _{1}A\right) \cap
\left( A\pi _{2}A\right) \neq \emptyset $, there exist $a_{i}\in A$ where $%
i=1,2,3,4$ and $x\in \pi _{1}$, $y\in \pi _{2}$ such that $%
a_{1}xa_{2}=a_{3}ya_{4}$. This implies $x=\left( a_{1}^{-1}a_{3}\right)
y\left( a_{4}a_{2}^{-1}\right) \in A\pi _{2}A$. Since $A$ is a union of
parts of $\Pi $, which is a d-partition, $A\pi _{2}A$ is a union of parts of 
$\Pi $ and hence, since $x\in \pi _{1}\in \Pi $ satisfies $x\in A\pi _{2}A$
we get that $\pi _{1}\subseteq A\pi _{2}A$. Multiplying by $A$ from each
side gives $A\pi _{1}A\subseteq A\pi _{2}A$. The inclusion $A\pi
_{2}A\subseteq A\pi _{1}A$ is proved similarly and hence $A\pi _{1}A=A\pi
_{2}A$.

Now we prove that the partition $\Pi ^{\prime }$ is a d-partition of $G$.
Clearly, $A\tau =\tau A=\tau $ for any $\tau \in \Pi ^{\prime }$ and hence $%
A $ acts as an identity. In order to complete the proof of the closure
property of $\Pi ^{\prime }$ it suffices to show that if $\pi _{1},\pi
_{2}\in \Pi \backslash \Pi _{<A}$ then $\left( A\pi _{1}A\right) \left( A\pi
_{2}A\right) $ is a union of parts of $\Pi ^{\prime }$. We have $\left( A\pi
_{1}A\right) \left( A\pi _{2}A\right) =A\left( \pi _{1}A\pi _{2}\right) A$.
Since $\Pi $ is a d-partition and $A$ is a union of parts of $\Pi $, we get
that $\pi _{1}A\pi _{2}$ is a union of parts of $\Pi $. Let $\pi \in \Pi $
be such a part. If $\pi \in \Pi _{<A}$ then $A\pi A=A\in \Pi ^{\prime }$.
Otherwise $\pi \in \Pi \backslash \Pi _{<A}$ and again $A\pi A\in \Pi
^{\prime }$.

Finally, $\left( A\pi A\right) ^{-1}=A\pi ^{-1}A$, and therefore the inverse
property of $\Pi $ implies that $\Pi ^{\prime }$ has the inverse property.

d. Note that by (c), $\Pi _{<A}$ is a d-partition of $A$, and hence, by (b),
we can assume $\Pi _{<A}=\left\{ A\right\} $. We prove the closure property
of $\Pi _{C}$ (the existence of an identity and the inverse property also
follow easily from the following arguments). We have $A^{2}=A$, and since $%
\pi _{c}$ is a union of double cosets of $A$, we have $A\pi _{c}=\pi
_{c}A=\pi _{c}$. It remains to consider $\pi _{c}^{2}$. If $\left\vert
G:A\right\vert =2$ then $\pi _{c}^{2}=A$, since $A\trianglelefteq G$ and $%
\pi _{c}$ is the single non-trivial coset of $A$. Otherwise $\left\vert
G:A\right\vert >2$. We prove that $\pi _{c}^{2}=G$. Notice that $\pi
_{c}^{-1}=\pi _{c}$, and that for any $x\in G$ it holds that $x\in \pi _{c}$
if and only if $Ax\subseteq \pi _{c}$. Let $h\in \pi _{c}$ be arbitrary.
Since $h^{-1}\in \pi _{c}$ and $Ah\subseteq \pi _{c}$, we have $A=\left(
Ah\right) h^{-1}\subseteq \pi _{c}^{2}$. Furthermore, since $\left\vert
G:A\right\vert >2$, there exists $g\in \pi _{c}$ such that $Ag^{-1}\neq
Ah^{-1}$ which is equivalent to $g^{-1}h\in \pi _{c}$. Therefore $Ah=\left(
Ag\right) \left( g^{-1}h\right) \subseteq \pi _{c}^{2}$. As $Ah$ is an
arbitrary coset of $A$ contained in $\pi _{c}$ we get $\pi _{c}\subseteq \pi
_{c}^{2}$, and therefore $\pi _{c}^{2}=G$.
\end{proof}

\begin{remark}
\label{Rem_T1}Part (c) of Theorem \ref{Th_General_1d_connection} is the same
as \cite[Proposition 1.4 ]{Tamaschke1968}. The translation of our
presentation in (c) to Tamaschke's terminology goes as follows. Let $T$ be
the S-semigroup associated with $\Pi $. The assumption that $\Pi _{<A}$ is a
partition of $A$ is equivalent in Tamaschke's terminology to the assumption
that $A$ is a $T$-subgroup of $G$ (see \cite[Definition 1.3]{Tamaschke1968}%
), and the conclusion that $\Pi _{<A}$ is a d-partition is equivalent to 
\cite[Proposition 1.4 (1)]{Tamaschke1968}. The claim concerning $\Pi
^{\prime }$ is equivalent to \cite[Proposition 1.4 (2)]{Tamaschke1968}.
\end{remark}

\begin{remark}
As a by-product of the proof of part (d) of Theorem \ref%
{Th_General_1d_connection}, we find that there are exactly two isomorphism
types (see Definition \ref{Def_IsomorphicDPartitions}) of $2$-part
d-partitions.
\end{remark}

\begin{proof}[Proof of Theorem \protect\ref{Th_CoarsenessViaAction}]
It is immediate to see that $\Pi ^{\prime }$ is a partition of $G$. Set $%
e:=e_{\Pi }$. Let $a\in A$. We have $1_{G}\in e^{a}\in \Pi $. This implies $%
e^{a}=e$ for all $a\in A$, and hence $e=U_{\left[ e\right] }\in \Pi ^{\prime
}$. Since every $U_{\left[ \pi \right] }$ is a union of parts of $\Pi $, we
get $eU_{\left[ \pi \right] }=U_{\left[ \pi \right] }e=U_{\left[ \pi \right]
}$. This proves $e_{\Pi ^{\prime }}=e$.

Now let $\pi _{1},\pi _{2}\in \Pi $ and consider the setwise product $U_{%
\left[ \pi _{1}\right] }U_{\left[ \pi _{2}\right] }$. We have:%
\begin{eqnarray*}
U_{\left[ \pi _{1}\right] }U_{\left[ \pi _{2}\right] }
&=&\tbigcup\limits_{a_{1},a_{2}\in A}\pi _{1}^{a_{1}}\pi
_{2}^{a_{2}}=\tbigcup\limits_{a_{1},a_{2}\in A}\left( \pi _{1}\pi
_{2}^{a_{2}a_{1}^{-1}}\right) ^{a_{1}}=\tbigcup\limits_{a\in
A}\tbigcup\limits_{g\in A}\left( \pi _{1}\pi _{2}^{g}\right) ^{a} \\
&=&\tbigcup\limits_{g\in A}\tbigcup\limits_{a\in A}\left( \pi _{1}\pi
_{2}^{g}\right) ^{a}\text{.}
\end{eqnarray*}%
Since $\pi _{1},\pi _{2}^{g}\in \Pi $ we have $\pi _{1}\pi
_{2}^{g}=\tbigcup\limits_{i\in I_{g}}\pi _{i}$, where $I_{g}$ is some
indexing set. It follows that 
\begin{gather*}
U_{\left[ \pi _{1}\right] }U_{\left[ \pi _{2}\right] }=\tbigcup\limits_{g\in
A}\tbigcup\limits_{a\in A}\left( \tbigcup\limits_{i\in I_{g}}\pi _{i}\right)
^{a}=\tbigcup\limits_{g\in A}\tbigcup\limits_{a\in A}\left(
\tbigcup\limits_{i\in I_{g}}\pi _{i}^{a}\right) =\tbigcup\limits_{g\in
A}\tbigcup\limits_{i\in I_{g}}\left( \tbigcup\limits_{a\in A}\pi
_{i}^{a}\right) = \\
=\tbigcup\limits_{g\in A}\tbigcup\limits_{i\in I_{g}}U_{\left[ \pi _{i}%
\right] }\text{.}
\end{gather*}%
Thus we have shown that $U_{\left[ \pi _{1}\right] }U_{\left[ \pi _{2}\right]
}$ can be written as a union of elements of $\Pi ^{\prime }$. The inverse
property follows immediately from $\left( \pi ^{a}\right) ^{-1}=\left( \pi
^{-1}\right) ^{a}$.
\end{proof}

\begin{corollary}
\label{Coro_AInvariantH}Let $G$ be a group and let $A$ be a group that acts
on $G$ as automorphisms. Let $H\leq G$ and $\Pi :=\left\{ HxH|x\in G\right\} 
$. Then $A$ acts on $\Pi $ if and only if $H$ is $A$-invariant, and in this
case $U_{\left[ HxH\right] }=HO_{A}\left( x\right) H$, where $O_{A}\left(
x\right) :=\tbigcup\limits_{a\in A}x^{a}$ is the orbit of $x$ under the
action of $A$.
\end{corollary}

\begin{proof}
Suppose that $A$ acts on $\Pi $. Then, by the proof of Theorem \ref%
{Th_CoarsenessViaAction}, $U_{\left[ H\right] }=H$ and hence $H$ is
invariant under the action of $A$. Conversely, suppose that $A$ acts on $H$.
Then, for any $x\in $ $G$ and $a\in A$ we get, using $H^{a}=H$, that $\left(
HxH\right) ^{a}=Hx^{a}H\in \Pi $, and hence $A$ acts on $\Pi $. In this case:%
\begin{equation*}
U_{\left[ HxH\right] }=\tbigcup\limits_{a\in A}\left( HxH\right)
^{a}=\tbigcup\limits_{a\in A}Hx^{a}H=H\left( \tbigcup\limits_{a\in
A}x^{a}\right) H=HO_{A}\left( x\right) H\text{.}
\end{equation*}
\end{proof}

\begin{example}
\label{Example_ZnMultiplicative}Let $n>0$ be an integer, $G=\left( \mathbb{Z}%
_{n},+_{n}\right) $, $H\leq G$ and $A\leq Aut\left( G\right) =\left( \mathbb{%
Z}_{n}^{\ast },\cdot _{n}\right) $. Since $G$ is abelian, each double coset
of $H$ in $G$ is a single coset of the form $H+x$ for some $x\in G$, and
since $G$ is cyclic, $H$ is a characteristic subgroup of $G$ and hence it is
invariant under $A$. Therefore, by Corollary \ref{Coro_AInvariantH}, $A$
acts on the d-partition $\Pi =\left\{ H+x|x\in G\right\} $ and the orbits of
this action form the d-partition $\Pi ^{\prime }=\left\{ H+A\cdot _{n}x|x\in
G\right\} $. In the special case where $n=p$ is a prime, non-trivial
examples arise from the choice $H=\left\{ 0\right\} $. For this choice $\Pi $
is the singleton partition, and $\Pi ^{\prime }=\left\{ \left\{ 0\right\}
\right\} \cup \left\{ A\cdot _{p}x|x\in \mathbb{Z}_{p}^{\ast }\right\} $,
i.e., the non-identity parts of $\Pi ^{\prime }$ are the cosets of the
multiplicative subgroup $A\leq \mathbb{Z}_{p}^{\ast }$, all of which have
size $\left\vert A\right\vert $.
\end{example}

The following is a consequence of Corollary \ref{Coro_AInvariantH}.

\begin{corollary}
\label{Coro_Supplement_construction}Let $G$ be a group, $N$ $\trianglelefteq 
$ $G$ and $A\leq G$ such that $AN=G$. Then:

\begin{enumerate}
\item[(a)] $\Pi _{N}:=\left\{ \left( A\cap N\right) O_{A}\left( n\right)
|n\in N\right\} $ is a d-partition of $N$ with $e_{\Pi _{N}}=A\cap N$.

\item[(b)] Let $\Pi _{G}:=\left\{ AyA|y\in G\right\} $. Then $f:\Pi
_{G}\rightarrow \Pi _{N}$ defined by $f\left( AyA\right) =\left( A\cap
N\right) O_{A}\left( n\right) $, where $n$ is any element of $N$ satisfying $%
yn^{-1}\in A$, is an isomorphism of d-partitions.
\end{enumerate}
\end{corollary}

In order to prove Corollary \ref{Coro_Supplement_construction} we begin with
a lemma.

\begin{lemma}
\label{Lem_DCtoAOrbitsCorresspondence}Let $G$ be a group, $N$ $%
\trianglelefteq $ $G$ and $A\leq G$ such that $AN=G$. Let $y\in G$. Then
there exists $n\in N$ such that $yn^{-1}\in A$ and for any such $n$ it holds
that 
\begin{equation*}
\left( AyA\right) \cap N=\left( A\cap N\right) O_{A}\left( n\right)
=O_{A}\left( n\right) \left( A\cap N\right) \text{.}
\end{equation*}
\end{lemma}

\begin{proof}
The existence of $n$ is clear from $AN=G$. Fix $n\in N$ for which $%
yn^{-1}\in A$. The equality $\left( A\cap N\right) O_{A}\left( n\right)
=O_{A}\left( n\right) \left( A\cap N\right) $ follows from the fact that $%
O_{A}\left( n\right) $ is normalized by $A$. It remains to prove $\left(
AyA\right) \cap N=O_{A}\left( n\right) \left( A\cap N\right) $. First we
prove $\left( AyA\right) \cap N\subseteq O_{A}\left( n\right) \left( A\cap
N\right) $. Let $g\in \left( AyA\right) \cap N=\left( AnA\right) \cap N$.
Then there exist $a_{1},a_{2}\in A$ such that $g=a_{1}^{-1}na_{2}=n^{a_{1}}%
\left( a_{1}^{-1}a_{2}\right) $. Note that $a_{1}^{-1}a_{2}\in A$ and $%
n^{a_{1}}\in N$. Since $g\in N$ we get $a_{1}^{-1}a_{2}=\left(
n^{a_{1}}\right) ^{-1}g\in N$. Therefore $a_{1}^{-1}a_{2}\in A\cap N$ and 
\begin{equation*}
g=n^{a_{1}}\left( a_{1}^{-1}a_{2}\right) \in n^{a_{1}}\left( A\cap N\right)
\subseteq O_{A}\left( n\right) \left( A\cap N\right) .
\end{equation*}%
Now we prove the reverse inclusion. Let $g\in O_{A}\left( n\right) \left(
A\cap N\right) $. Since both $O_{A}\left( n\right) $ and $\left( A\cap
N\right) $ are subsets of $N$ we get $g\in N$, and it remains to prove $g\in
AyA$. There exist $a\in A$ and $n^{\prime }\in A\cap N$ such that 
\begin{equation*}
g=n^{a}n^{\prime }=a^{-1}nan^{\prime }\text{.}
\end{equation*}%
But $an^{\prime }\in A$ since $n^{\prime }\in A\cap N$ so $g\in AnA=AyA$ as
required.
\end{proof}

\begin{proof}[Proof of Corollary \protect\ref{Coro_Supplement_construction}]

(a) First note that the elements of $A$ act as automorphisms on $N$ via $%
n^{a}=a^{-1}na$ for all $a\in A$ and $n\in N$. Let $\Pi _{N}^{\prime
}:=\left\{ \left( A\cap N\right) n\left( A\cap N\right) |n\in N\right\} $ be
the set of double cosets of $A\cap N\leq N$. Then $\Pi _{N}^{\prime }$ is a
d-partition of $N$, and since $A$ leaves $A\cap N$ invariant, $A$ acts on $%
\Pi _{N}^{\prime }$ by Corollary \ref{Coro_AInvariantH}. Let $n\in N$ be
arbitrary and set $\pi :=\left( A\cap N\right) n\left( A\cap N\right) $. By
Corollary \ref{Coro_AInvariantH} and Lemma \ref%
{Lem_DCtoAOrbitsCorresspondence}, we get%
\begin{equation*}
U_{\left[ \pi \right] }=\left( A\cap N\right) O_{A}\left( n\right) \left(
A\cap N\right) =\left( A\cap N\right) ^{2}O_{A}\left( n\right) =\left( A\cap
N\right) O_{A}\left( n\right) \text{,}
\end{equation*}%
and now claim (a) follows from Theorem \ref{Th_CoarsenessViaAction}.

(b) The map defined by $AyA\longmapsto \left( AyA\right) \cap N$ from $\Pi
_{G}$ into the power set of $N$ is clearly well-defined. By Lemma \ref%
{Lem_DCtoAOrbitsCorresspondence} it is equal to $f$. Hence $f$ is a
well-defined $\Pi _{G}\rightarrow \Pi _{N}$ map. Its surjectivity follows
from $G=\tbigcup\limits_{y\in G}AyA$ and its injectivity follows from the
fact that distinct double cosets of $A$ have empty intersection. It remains
to prove that $f$ preserves the structure constants. Let $y,y^{\prime }\in G$
and $n,n^{\prime }\in N$ be arbitrary such that $yn^{-1},y^{\prime }\left(
n^{\prime }\right) ^{-1}\in A$. It would suffice to prove that%
\begin{equation*}
N\cap \left( \left( AyA\right) \left( Ay^{\prime }A\right) \right) =\left(
\left( A\cap N\right) O_{A}\left( n\right) \right) \left( \left( A\cap
N\right) O_{A}\left( n^{\prime }\right) \right) \text{.}
\end{equation*}%
We have%
\begin{equation*}
\left( AyA\right) \left( Ay^{\prime }A\right) =AnAn^{\prime }A=A\left(
\tbigcup\limits_{a\in A}aa^{-1}nan^{\prime }\right) A=\tbigcup\limits_{a\in
A}An^{a}n^{\prime }A\text{.}
\end{equation*}%
By Lemma \ref{Lem_DCtoAOrbitsCorresspondence} we have $N\cap \left(
An^{a}n^{\prime }A\right) =O_{A}\left( n^{a}n^{\prime }\right) \left( A\cap
N\right) $. Therefore,%
\begin{gather*}
N\cap \left( AyA\right) \left( Ay^{\prime }A\right) =N\cap
\tbigcup\limits_{a\in A}An^{a}n^{\prime }A=\tbigcup\limits_{a\in A}N\cap
\left( An^{a}n^{\prime }A\right) = \\
=\tbigcup\limits_{a\in A}O_{A}\left( n^{a}n^{\prime }\right) \left( A\cap
N\right) =\left( A\cap N\right) \tbigcup\limits_{a\in A}O_{A}\left(
n^{a}n^{\prime }\right) \text{.}
\end{gather*}%
Now 
\begin{gather*}
\tbigcup\limits_{a\in A}O_{A}\left( n^{a}n^{\prime }\right)
=\tbigcup\limits_{a\in A}\tbigcup\limits_{b\in A}\left( n^{a}n^{\prime
}\right) ^{b}=\tbigcup\limits_{a\in A}\tbigcup\limits_{b\in A}n^{ab}\left(
n^{\prime }\right) ^{b} \\
=\tbigcup\limits_{a_{1}\in A}\tbigcup\limits_{a_{2}\in A}n^{a_{1}}\left(
n^{\prime }\right) ^{a_{2}}=O_{A}\left( n\right) O_{A}\left( n^{\prime
}\right) \text{.}
\end{gather*}%
Hence, since $A\cap N=\left( A\cap N\right) ^{2}$ commutes with $O_{A}\left(
n\right) $ and $O_{A}\left( n^{\prime }\right) $ we get%
\begin{equation*}
N\cap \left( AyA\right) \left( Ay^{\prime }A\right) =\left( \left( A\cap
N\right) O_{A}\left( n\right) \right) \left( \left( A\cap N\right)
O_{A}\left( n^{\prime }\right) \right) \text{.}
\end{equation*}
\end{proof}

\begin{remark}
\label{Rem_T4}Corollary \ref{Coro_Supplement_construction} can also be
proved using Tamaschke's second isomorphism theorem 2.13 in \cite%
{Tamaschke1968} as follows (we use the notation of \cite{Tamaschke1968}).
Denote by $T^{\left( G\right) }$ and by $T^{\left( N\right) }$ the
S-semigroups on $G$ and on $N$ associated with $\Pi _{G}$ and $\Pi _{N}$
respectively. Since $A\cap N\trianglelefteq A$, we can apply Corollary \ref%
{Coro_AInvariantH} with $H=A\cap N$ and obtain that $\Upsilon :=\left\{
\left( A\cap N\right) O_{A}\left( x\right) \left( A\cap N\right) |x\in
G\right\} $ is a d-partition of $G$. Let $T$ be the corresponding
S-semigroup. We observe that $A$ and $N$ are $T$-subgroups of $G$ with%
\begin{eqnarray*}
T_{A} &=&\left[ \left( A\cap N\right) O_{A}\left( a\right) \left( A\cap
N\right) |a\in A\right] \\
T_{N} &=&\left[ \left( A\cap N\right) O_{A}\left( n\right) \left( A\cap
N\right) |n\in N\right] =T^{\left( N\right) }\text{.}
\end{eqnarray*}%
Furthermore, $N$ is $T$-normal since $N\trianglelefteq G$, and $A$ is $T$%
-normal since it normalizes both $A\cap N$ and $O_{A}\left( x\right) $ for
any $x\in G$. Now we can apply \cite[Theorem 2.13]{Tamaschke1968}, with $K=A$
and $H=N$. By \cite[Theorem 2.13]{Tamaschke1968} (3) we get 
\begin{gather*}
T_{H}/T_{H\cap K}=T_{N}/T_{N\cap A}=T_{N}=T^{\left( N\right) } \\
T_{HK}/T_{K}=T/T_{A}=\left[ AyA|y\in A\right] =T^{\left( G\right) }\text{,}
\end{gather*}%
and the isomorphism claim follows.
\end{remark}

Next we prove that d-partitions lift from quotient groups.

\begin{proof}[Proof of Theorem \protect\ref{Th_LiftFromQuotient}]
Let $\nu :G\rightarrow \overline{G}$ be the natural projection and for any $%
\overline{S}\subseteq \overline{G}$ let $\nu ^{-1}\left( \overline{S}\right) 
$ be the full preimage of $\overline{S}$ in $G$. Recall that $S\subseteq G$
is the full preimage of $\overline{S}$ in $G$ if and only if $\nu \left(
S\right) =\overline{S}$\ and $NS=S$. For any $\overline{\pi }\in $ $%
\overline{\Pi }$, denote $\pi :=\nu ^{-1}\left( \overline{\pi }\right) $.
Now let $\overline{\pi }\in $ $\overline{\Pi }$. Then $N\pi =\pi $ implies
that $\pi $ is a union of cosets of $N$. Since two distinct cosets of $N$
are disjoint, and since $\overline{\pi _{1}}\neq \overline{\pi _{2}}$
implies $\overline{\pi _{1}}\cap \overline{\pi _{2}}=\emptyset $, we get
that $\overline{\pi _{1}}\neq \overline{\pi _{2}}$ implies $\pi _{1}\cap \pi
_{2}=\emptyset $. Since $\pi $ is the full preimage of $\overline{\pi }$, $%
\overline{G}=\tbigcup\limits_{\overline{\pi }\in \overline{\Pi }}\overline{%
\pi }$ implies $G=\tbigcup\limits_{\pi \in \Pi }\pi $. Therefore $\Pi $ is a
partition of $G$.

Let $\pi _{1},\pi _{2}\in \Pi $. Then $\nu \left( \pi _{1}\pi _{2}\right)
=\nu \left( \pi _{1}\right) \nu \left( \pi _{2}\right) $ is a union of parts
of $\overline{\Pi }$. Hence $\nu ^{-1}\left( \nu \left( \pi _{1}\pi
_{2}\right) \right) $ is a union of parts of $\Pi $. But $N\pi _{1}=\pi
_{1}\ $implies $N\left( \pi _{1}\pi _{2}\right) =\pi _{1}\pi _{2}$ so $\pi
_{1}\pi _{2}=\nu ^{-1}\left( \nu \left( \pi _{1}\pi _{2}\right) \right) $.
This proves that $\Pi $ is closed under multiplication. Furthermore, $e_{\Pi
}=\nu ^{-1}\left( e_{\overline{\Pi }}\right) $ and since for any $\pi \in
\Pi $, $\pi ^{-1}=\nu ^{-1}\left( \overline{\pi }^{-1}\right) $, we get the
inverse property of $\Pi $ from that of $\overline{\Pi }$.
\end{proof}

We note that the converse of Theorem \ref{Th_LiftFromQuotient} is not true,
namely, if $G$ is a group, $N\trianglelefteq G$, $\Pi $ is a d-partition of $%
G$, and $\nu :G\rightarrow \overline{G}:=G/N$ is the natural projection,
then $\left\{ \nu \left( \pi \right) |\pi \in \Pi \right\} $ is not
necessarily a d-partition. An easy counterexample is provided by $G=\left( 
\mathbb{Z},+\right) $, $N=2\mathbb{Z}$ and $\Pi :=\left\{ \left\{ 0\right\} ,%
\mathbb{Z}\backslash \left\{ 0\right\} \right\} $. Clearly, $\nu \left(
\left\{ 0\right\} \right) =\left\{ 0\right\} $ and $\nu \left( \mathbb{Z}%
\backslash \left\{ 0\right\} \right) =\left\{ 0,1\right\} $ which are
distinct but have a non-empty intersection. Nevertheless, there are examples
of d-partitions of groups which are mapped nicely to quotients, e.g., double
coset partitions.

\section{Three part d-partitions of $\left( \mathbb{Z}_{p},+_{p}\right) $%
\label{Sect_DPartitonsOfZp}}

In this section we study d-partitions of the form $\Pi =\left\{ \pi _{0},\pi
_{1},\pi _{2}\right\} $, of $\left( \mathbb{Z}_{p},+_{p}\right) $, where $p$
is a prime. Note that one can assume, without loss of generality, $\pi
_{0}:=\left\{ 0\right\} $, and $\left\vert \pi _{1}\right\vert \leq
\left\vert \pi _{2}\right\vert $. Our results indicate that these
d-partitions involve interesting structures related to fundamental questions
in additive combinatorics (see Section \ref{Subsect_AdditiveCombResults} for
the concepts used in the analysis). Recall, in comparison, that the
S-partitions of these groups were fully classified in \cite%
{BGordonPacific1964}\ (see Theorem \ref{Th_BGordon}).

\subsection{Isomorphism Types of $3$-part d-partitions of $\left( \mathbb{Z}%
_{p},+_{p}\right) $. \label{Subsect_StructureConstsFor3part}}

We begin with determining the isomorphism types of the $3$-part d-partitions
of $\left( \mathbb{Z}_{p},+_{p}\right) $. Two "small $p$" cases are
summarized in the following remark.

\begin{remark}
\label{Rem_SmallpIsos}One can check that the only possibilities for $\Pi $
for the primes $3$ and $5$ are the singleton partition if $p=3$ and $\Pi
=\left\{ \left\{ 0\right\} ,\left\{ 1,4\right\} ,\left\{ 2,3\right\}
\right\} $ if $p=5$. It turns out that the isomorphism types of these two
d-partitions do not repeat for larger primes.
\end{remark}

From now on we assume $p>5$.

\begin{lemma}
\label{Lem_StructureConstsFor3Parts}Let $p>5$ be a prime, $G=\left( \mathbb{Z%
}_{p},+_{p}\right) $, and $\Pi =\left\{ \pi _{0},\pi _{1},\pi _{2}\right\} $
a $3$-part d-partition of $G$, where $\pi _{0}=\left\{ 0\right\} $ and $%
\left\vert \pi _{1}\right\vert \leq \left\vert \pi _{2}\right\vert $. Then
one of the following \emph{(a)} and \emph{(b)} holds true:

\begin{enumerate}
\item[(a)] $\pi _{k}=-\pi _{k}$ for $k=1,2$, $\pi _{1}+\pi _{2}=\pi _{1}\cup
\pi _{2}$ and $\pi _{2}+\pi _{2}=G$. In addition, either:
\end{enumerate}

\qquad \emph{(a1)} $\pi _{1}+\pi _{1}=\pi _{0}\cup \pi _{2}$, or:

\qquad \emph{(a2)} $\pi _{1}+\pi _{1}=G$.

\begin{enumerate}
\item[(b)] $\pi _{1}=-\pi _{2}$, $\pi _{1}+\pi _{2}=G$, and $\pi _{1}+\pi
_{1}=\pi _{2}+\pi _{2}=\pi _{1}\cup \pi _{2}$.
\end{enumerate}
\end{lemma}

\begin{proof}

(i) Let $i,j\in \left\{ 1,2\right\} $. Theorem \ref{Th_CauchyDavenport} gives%
\begin{equation}
\left\vert \pi _{i}+\pi _{j}\right\vert \geq \min \left( \left\vert \pi
_{i}\right\vert +\left\vert \pi _{j}\right\vert -1,p\right) \geq 2\left\vert
\pi _{1}\right\vert -1\text{,}  \label{Eq_CDFor3Parts}
\end{equation}%
where for the second inequality we used $\left\vert \pi _{1}\right\vert \leq
\left\vert \pi _{2}\right\vert $.

(ii) We prove $\pi _{2}\subseteq \pi _{i}+\pi _{j}$. Assume by contradiction
that this is not the case. Since by the closure property of a d-partition $%
\pi _{i}+\pi _{j}$ is a union of parts, we get $\pi _{i}+\pi _{j}\subseteq
\left\{ 0\right\} \cup \pi _{1}$ which gives, using Equation (\ref%
{Eq_CDFor3Parts}), $2\left\vert \pi _{1}\right\vert -1\leq \left\vert \pi
_{1}\right\vert +1$. This implies $\left\vert \pi _{1}\right\vert \leq 2$,
and hence, by Lemma \ref{Lem_ZpSingletonOrDoubleImplyMult}, $\left\vert \pi
_{1}\right\vert =\left\vert \pi _{2}\right\vert \leq 2$ in contradiction to $%
p>5$.

(iii) Suppose that at least one of $i$ and $j$ is not equal to $1$. Note
that under this assumption the first inequality in (\ref{Eq_CDFor3Parts})
gives $\left\vert \pi _{i}+\pi _{j}\right\vert \geq p-2$. We prove that $\pi
_{1}\subseteq \pi _{i}+\pi _{j}$. Assume by contradiction $\left( \pi
_{i}+\pi _{j}\right) \cap \pi _{1}=\emptyset $. Then $\pi _{i}+\pi
_{j}\subseteq \left\{ 0\right\} \cup \pi _{2}$ and we get $p-2\leq
\left\vert \pi _{i}+\pi _{j}\right\vert \leq \left\vert \pi _{2}\right\vert
+1$ implying $\left\vert \pi _{1}\right\vert \leq 2$. Applying Lemma \ref%
{Lem_ZpSingletonOrDoubleImplyMult}, we obtain a contradiction as in (ii).

(iv) Now we are ready to prove the claim of the lemma. By the inverse
property of a d-partition and $\left\vert \Pi \right\vert =3$, either $\pi
_{k}=-\pi _{k}$ for $k=1,2$ or $\pi _{1}=-\pi _{2}$. Suppose $\pi _{k}=-\pi
_{k}$ for $k=1,2$. For any $x\in \pi _{1}$ and $y\in \pi _{2}$, $y\neq -x$
and hence $x+y\neq 0$. Hence $0\notin \pi _{1}+\pi _{2}$ and $\pi _{1}+\pi
_{2}=\pi _{1}\cup \pi _{2}$ by (ii) and (iii). On the other hand, $0\in \pi
_{2}+\pi _{2}$ and therefore $\pi _{2}+\pi _{2}=G$. Similarly, $0\in \pi
_{1}+\pi _{1}$ and hence either $\pi _{1}+\pi _{1}=\pi _{0}\cup \pi _{2}$,
which gives (a1), or $\pi _{1}+\pi _{1}=G$, which gives (a2). Now suppose
that $\pi _{1}=-\pi _{2}$. Here $0\in \pi _{1}+\pi _{2}$ and hence $\pi
_{1}+\pi _{2}=G$. On the other hand, $0\notin \pi _{i}+\pi _{i}$ where $%
i=1,2 $. Since $\pi _{1}=-\pi _{2}$ we have $\left\vert \pi _{1}\right\vert
=\left\vert \pi _{2}\right\vert $ so in this case we can prove $\pi
_{2}\subseteq \pi _{1}+\pi _{1}$ using the same reasoning as in (iii). Hence
we get $\pi _{1}+\pi _{1}=\pi _{2}+\pi _{2}=\pi _{1}\cup \pi _{2}$. This
gives case (b) of the lemma.
\end{proof}

We can recast Lemma \ref{Lem_StructureConstsFor3Parts} in terms of the
structure constants $d_{i,j}^{k}$ of the induced dioids (see Lemma \ref%
{Lem_Def_DioidStructureConst} and Remark \ref{Rem_01}). We have $\pi
_{i}+\pi _{j}=\tbigcup\limits_{k=0}^{2}d_{i,j}^{k}\pi _{k}$, where $i,j\in
\left\{ 0,1,2\right\} $. It is sufficient to consider $i\leq j\in \left\{
1,2\right\} $.

\begin{corollary}
\label{Coro_StructureConstsFor3Parts}Under the assumptions and notation of
Lemma \ref{Lem_StructureConstsFor3Parts}:

\begin{enumerate}
\item[(a)] If $\pi _{1}=-\pi _{1}$ then 
\begin{eqnarray*}
d_{1,1}^{0}
&=&d_{1,1}^{2}=d_{1,2}^{1}=d_{1,2}^{2}=d_{2,2}^{0}=d_{2,2}^{1}=d_{2,2}^{2}=1%
\text{,} \\
d_{1,2}^{0} &=&0\text{,}
\end{eqnarray*}%
while $d_{1,1}^{1}$ is either $0$ or $1$.

\item[(b)] If $\pi _{1}=-\pi _{2}$ then 
\begin{eqnarray*}
d_{1,1}^{1}
&=&d_{1,1}^{2}=d_{1,2}^{0}=d_{1,2}^{1}=d_{1,2}^{2}=d_{2,2}^{1}=d_{2,2}^{2}=1%
\text{,} \\
d_{1,1}^{0} &=&d_{2,2}^{0}=0\text{.}
\end{eqnarray*}
\end{enumerate}
\end{corollary}

We are ready to prove Theorem \ref{Th_3PartIsoClassification} which
classifies the isomorphism types of the $3$-part d-partitions of a cyclic
group of prime order.

\begin{proof}[Proof of Theorem \protect\ref{Th_3PartIsoClassification}]
Let $\Pi =\left\{ \pi _{0},\pi _{1},\pi _{2}\right\} $, with $\pi
_{0}=\left\{ 0\right\} $ and $\left\vert \pi _{1}\right\vert \leq \left\vert
\pi _{2}\right\vert $, be a $3$-part partition of $G$. By Remark \ref%
{Rem_SmallpIsos} we may assume $p>5$. If $\Pi $ is a d-partition then, by
Lemma \ref{Lem_StructureConstsFor3Parts}, exactly one of (3)-(5) holds (note
that if $\pi _{2}$ is an arithmetic progression of size $\frac{p-1}{2}$ then 
$\left\vert \pi _{2}+\pi _{2}\right\vert =p-2$ and hence $\pi _{2}+\pi
_{2}\neq G$).

We now prove that the conditions stated in (3)-(5) are sufficient for $\Pi $
to be a d-partition. Clearly $\pi _{0}$ is an identity element and $\Pi $
has the inversion property since $\pi _{1}=-\pi _{1}$ implies $\pi _{2}=-\pi
_{2}$ and $\pi _{1}=-\pi _{2}$ implies $\pi _{2}=-\pi _{1}$. We now prove
the closure property. Note that $\left\vert \pi _{1}\right\vert \leq
\left\vert \pi _{2}\right\vert $ implies 
\begin{equation*}
2\left\vert \pi _{2}\right\vert -1\geq \left\vert \pi _{1}\right\vert
+\left\vert \pi _{2}\right\vert -1=p-2\text{.}
\end{equation*}%
Hence, by Theorem \ref{Th_CauchyDavenport} we have:%
\begin{equation}
\left\vert \pi _{1}+\pi _{2}\right\vert \geq \min \left( \left\vert \pi
_{1}\right\vert +\left\vert \pi _{2}\right\vert -1,p\right) =\left\vert \pi
_{1}\right\vert +\left\vert \pi _{2}\right\vert -1=p-2\text{,}  \tag{$\ast $}
\end{equation}%
and 
\begin{equation}
\left\vert \pi _{2}+\pi _{2}\right\vert \geq \min \left( 2\left\vert \pi
_{2}\right\vert -1,p\right) \geq p-2\text{.}  \tag{$\ast \ast $}
\end{equation}

\begin{enumerate}
\item Suppose that either condition (3) or (4) of the theorem holds. Since $%
\pi _{1}=-\pi _{1}$ we have that $\pi _{1}+\pi _{2}$ is a symmetric set, and
also $0\notin \pi _{1}+\pi _{2}$, which implies by $\left( \ast \right) $
that $\left\vert \pi _{1}+\pi _{2}\right\vert \in \left\{ p-2,p-1\right\} $.
Since $p$ is odd, $x\neq -x$ for any $0\neq x\in G$. Hence, any symmetric
subset of $G$ which does not contain $0$ has even cardinality, and this
forces $\left\vert \pi _{1}+\pi _{2}\right\vert =p-1$. Hence $\pi _{1}+\pi
_{2}=\pi _{1}\cup \pi _{2}$. It remains to prove that $\pi _{2}+\pi _{2}=G$.
Assume by contradiction that $\pi _{2}+\pi _{2}\subset G$. By $\left( \ast
\ast \right) $ this implies $\left\vert \pi _{2}+\pi _{2}\right\vert \in
\left\{ p-2,p-1\right\} $. Since $\pi _{2}=-\pi _{2}$, we get that $\pi
_{2}+\pi _{2}$ is a symmetric set, and that $0\in \pi _{2}+\pi _{2}$.
Therefore $\left\vert \pi _{2}+\pi _{2}\right\vert $ is odd, so $\left\vert
\pi _{2}+\pi _{2}\right\vert =p-2$. Now $\left( \ast \ast \right) $ gives $%
2\left\vert \pi _{2}\right\vert -1=p-2$, and $\left\vert \pi _{1}\right\vert
=\left\vert \pi _{2}\right\vert =\frac{p-1}{2}$. Applying Theorem \ref%
{Th_Vosper} with $A=B=\pi _{2}$ we get that $\pi _{2}$ is an arithmetic
progression of size $\frac{p-1}{2}$. Hence, condition (3) of the theorem
holds, giving $\pi _{1}+\pi _{1}=\pi _{0}\cup \pi _{2}$, implying $%
\left\vert \pi _{1}+\pi _{1}\right\vert =\left\vert \pi _{2}\right\vert +1$,
but, by Theorem \ref{Th_CauchyDavenport}, $\left\vert \pi _{1}+\pi
_{1}\right\vert \geq 2\left\vert \pi _{1}\right\vert -1$. Hence, using $%
\left\vert \pi _{1}\right\vert =\left\vert \pi _{2}\right\vert $ we get $%
\left\vert \pi _{2}\right\vert \leq 2$ in contradiction to $p>5$.

\item Suppose that condition (5) holds, namely, $\pi _{1}=-\pi _{2}$ and $%
\pi _{1}+\pi _{1}=\pi _{1}\cup \pi _{2}$. By negating both sides we get $\pi
_{2}+\pi _{2}=\pi _{1}\cup \pi _{2}$. We now prove $\pi _{1}+\pi _{2}=G$.
Note that $\pi _{1}=-\pi _{2}$ implies $0\in \pi _{1}+\pi _{2}$ and also
that $\pi _{1}+\pi _{2}$ is symmetric. This implies that $\left\vert \pi
_{1}+\pi _{2}\right\vert $ is odd. Thus, by $\left( \ast \right) $, $%
\left\vert \pi _{1}+\pi _{2}\right\vert \in \left\{ p-2,p\right\} $. Suppose 
$\left\vert \pi _{1}+\pi _{2}\right\vert =p-2$. Then, by $\left( \ast
\right) $, $\left\vert \pi _{1}+\pi _{2}\right\vert =\left\vert \pi
_{1}\right\vert +\left\vert \pi _{2}\right\vert -1$. By Theorem \ref%
{Th_Vosper} with $A=\pi _{1}$ and $B=\pi _{2}$, we get that $\pi _{1}$ is an
arithmetic progression of length $\left\vert \pi _{1}\right\vert =\frac{p-1}{%
2}$. Hence $\left\vert \pi _{1}+\pi _{1}\right\vert =2\left\vert \pi
_{1}\right\vert -1<\left\vert \pi _{1}\right\vert +\left\vert \pi
_{2}\right\vert $, in contradiction to $\pi _{1}+\pi _{1}=\pi _{1}\cup \pi
_{2}$. Therefore $\left\vert \pi _{1}+\pi _{2}\right\vert =p$, implying $\pi
_{1}+\pi _{2}=G$.
\end{enumerate}
\end{proof}

\begin{remark}
\label{Rem_sIsoTypes}Using Corollary \ref%
{Coro_MultiplicativePartitionsUnderInverse} it is easy to check that the $3$%
-part S-partitions of $\left( \mathbb{Z}_{p},+_{p}\right) $, where $p$ is a
prime, fit in the classification described in Theorem \ref%
{Th_3PartIsoClassification} as follows.

\begin{description}
\item[1] Types (1) and (2) are S-partitions.

\item[2] For $p\equiv 1\left( \func{mod}4\right) $, $p>5$, a $3$-part
S-partition is of type (4).

\item[3] For $p\equiv 3\left( \func{mod}4\right) $, $p>5$, a $3$-part
S-partition is of type (5).
\end{description}
\end{remark}

\begin{remark}
\label{Rem_d-partitionsWhichAreNotS-partitions}For every prime $p\leq 7$,
every $3$-part d-partition is an S-partition. This can be verified using
Lemma \ref{Lem_ZpSingletonOrDoubleImplyMult} and some further explicit, easy
calculations. On the other hand, for every prime $p\geq 11$, there exists a $%
3$-part d-partition which is not an S-partition. To see this note that by
Remark \ref{Rem_sIsoTypes} there is no $3$-part S-partition of type (3),
while the discussion in Section \ref{Subsect_SolutionsTo1} shows that $3$%
-part d-partitions of type (3) exist for every $p\geq 11$. Notice that this
observation answers (in the negative) Tamaschke's Problem 1.19 in \cite%
{Tamaschke1969}.
\end{remark}

\subsection{Constructions of $3$-part d-partitions of $\left( \mathbb{Z}%
_{p},+_{p}\right) $\label{Sect_ConstructionsOf3part}}

Let $p$ be a prime and $G=\left( \mathbb{Z}_{p},+_{p}\right) $. In this
section we consider explicit constructions of $3$-part d-partitions of $G$.

By Theorem \ref{Th_3PartIsoClassification}, every $3$-part d-partition of $G$%
, where $p>5$ is a prime, is of the form $\Pi =\left\{ \left\{ 0\right\} ,S,%
\overline{S\cup \left\{ 0\right\} }\right\} $ where $S$, $0<\left\vert
S\right\vert \leq \frac{p-1}{2}$, is a solution to one of the following
three equations:%
\begin{equation}
\text{ }S+S=\overline{S},~S=-S\text{ }  \label{Eq_pi1_1}
\end{equation}%
\begin{gather}
\text{ }S+S=G,~S=-S\text{, }0\notin S\text{, and }  \label{Eq_pi1_2} \\
\overline{S\cup \left\{ 0\right\} }\text{ is not an arithmetic progression
of size }\frac{p-1}{2}\text{ }  \notag
\end{gather}%
\begin{equation}
\text{ }S+S=\overline{\left\{ 0\right\} },~S=-\overline{S\cup \left\{
0\right\} }\text{.}  \label{Eq_pi1_3}
\end{equation}

In the following we discuss separately the solutions to each of these
equations. Equation (\ref{Eq_pi1_1}) had already been studied in the
literature in other contexts. Hence, in Section \ref{Subsect_SolutionsTo1}
(to follow) we review the relevant results and explain their connection to
our settings. In sections \ref{Subsect_SolutionsTo2} and \ref%
{Subsect_SolutionsTo3} we provide a full classification of the solutions to
Equations (\ref{Eq_pi1_2}) and (\ref{Eq_pi1_3}).

\subsubsection{Solutions to Equation (\protect\ref{Eq_pi1_1}) \label%
{Subsect_SolutionsTo1}}

Equation (\ref{Eq_pi1_1}) fits into a well-studied topic in additive
combinatorics - the theory of sum-free sets. An additive set $S$ is called
sum-free if $S+S\subseteq \overline{S}$, and complete if $\overline{S}%
\subseteq S+S$. Thus, $S$ is a solution to Equation (\ref{Eq_pi1_1}) if and
only if $S$ is symmetric, sum-free and complete. Note that all of these
properties are invariant under group automorphisms. This part briefly
reviews some known results which are relevant to the description of the
associated $3$-part d-partitions.

Let $p\geq 5$ be a prime. The symmetric, sum-free and complete subsets of $%
G=\left( \mathbb{Z}_{p},+_{p}\right) $ of largest possible cardinality, can
be deduced from the classification of the largest cardinality sum-free sets
by Diananda and Yap \cite{Yap1968,Yap1970} for $p=3k+2$ and by Rhemtulla and
Street \cite{RhemtullaStreet} for $p=3k+1$. They are uniquely given, up to
an automorphism, as follows: For $p=3k+2$\ \ it is the interval $\left[
k+1,2k+1\right] $ of cardinality $k+1$ and for $p=3k+1$ with $k\geq 4$ it is
the set $\left\{ k\right\} \cup \left[ k+2,2k-1\right] \cup \left\{
2k+1\right\} $ of cardinality $k$. In a recent paper \cite{HavivLevySF2017},
whose original motivation was the understanding of $3$-part d-partitions, we
extended this result and characterized all symmetric complete sum-free sets
of size at least $c\cdot p$ where $c=0.318$ and $p$ is a sufficiently large
prime, and proved that their number grows exponentially in $p$. Moreover, we
have shown that there exist constants $c_{1},c_{2},c_{3}>0$ such that for
every sufficiently large integer $n$ there exists a collection of symmetric
complete sum-free subsets of $\mathbb{Z}_{n}$ whose sizes form an arithmetic
progression with first element at most $c_{1}\sqrt{n}$, step at most $c_{2}%
\sqrt{n}$, and last element at least $\frac{n}{3}-c_{3}\sqrt{n}$ (see \cite%
{HavivLevySF2017} for more details). In particular, there exists a symmetric
complete sum-free subset of $\mathbb{Z}_{n}$ of size proportional to $\sqrt{n%
}$.

\subsubsection{Solutions to Equation (\protect\ref{Eq_pi1_2}) \label%
{Subsect_SolutionsTo2}}

Let $p>5$ be a prime. We characterize the solutions\ $S$ of Equation (\ref%
{Eq_pi1_2}) with $0<\left\vert S\right\vert \leq \frac{p-1}{2}$. It suffices
to characterize the symmetric sets $S\subseteq \mathbb{Z}_{p}^{\ast }$ such
that $0<\left\vert S\right\vert \leq \frac{p-1}{2}$ and $S+S\neq \mathbb{Z}%
_{p}$. For this we need the following definition.

\begin{definition}
\label{Def_avoiding}Let $p$ be a prime and let $S\subseteq \mathbb{Z}%
_{p}^{\ast }$ be a non-empty set of size $s:=\left\vert S\right\vert \leq 
\frac{p-1}{2}$. We say that $S$ is \emph{avoiding} if $S$ is symmetric and%
\textit{\ can be written in the form}%
\begin{equation}
S=\left\{ \pm i_{1},\pm i_{2},...,\pm i_{s/2}\right\} \text{,}
\label{Eq_Avoiding}
\end{equation}%
where $1\leq i_{1}<i_{2}<\cdots <i_{s}\leq \frac{p-1}{2}$, and $i_{j+1}\geq
i_{j}+2$ for all $1\leq j\leq \frac{s}{2}$, where $i_{\frac{s}{2}+1}:=p-i_{%
\frac{s}{2}}$.
\end{definition}

Note that an avoiding set must have a positive even size, and hence such
sets exist only for primes $p\geq 5$.

\begin{theorem}
Let $p$ be a prime. Let $S\subseteq \mathbb{Z}_{p}^{\ast }$ be a non-empty
symmetric set of size $s:=\left\vert S\right\vert \leq \frac{p-1}{2}$. Then $%
S$ is \textit{avoiding, up to an automorphism, if and only if }$S+S\neq 
\mathbb{Z}_{p}$.
\end{theorem}

\begin{proof}
a. Assume that, \textit{up to an automorphism,} $S$ is avoiding. Since the
conditions $S=-S$ and $S+S\neq \mathbb{Z}_{p}$ are invariant under
automorphisms, we may assume that $S$ is avoiding. By Definition \ref%
{Def_avoiding}, no two elements of $S$ are consecutive $\func{mod}p$, and
hence $1\notin S-S=S+S$. Therefore $S+S\neq \mathbb{Z}_{p}$.\textit{\ }

b. Assume that $S+S\neq \mathbb{Z}_{p}$. Since $S\neq \emptyset $ and $S$ is
symmetric, we have $0\in S+S$. Hence, $S+S\neq \mathbb{Z}_{p}$ implies the
existence of $x\in \mathbb{Z}_{p}^{\ast }$ such that $x\notin S+S$. Then $%
S=x\left( x^{-1}S\right) $ where $x^{-1}S$ is a symmetric subset of $\mathbb{%
Z}_{p}^{\ast }$ of cardinality $s$. Hence $x^{-1}S=\left\{ \pm i_{1},\pm
i_{2},...,\pm i_{s/2}\right\} $, where $1\leq i_{1}<i_{2}<\cdots <i_{s}\leq 
\frac{p-1}{2}$. Suppose that $i_{j+1}=i_{j}+1$, for some $1\leq j\leq \frac{s%
}{2}$, where $i_{\frac{s}{2}+1}:=p-i_{\frac{s}{2}}$. Then $%
x=xi_{j+1}+x\left( -i_{j}\right) \in S+S$ - a contradiction. Hence $%
i_{j+1}\geq i_{j}+2$, for all $1\leq j\leq \frac{s}{2}$, and $x^{-1}S$ is
avoiding.
\end{proof}

The following proposition uses elementary combinatorics for counting
avoiding sets of a given size.

\begin{proposition}
\label{Prop_NumOfAvoiding}For a prime $p\geq 5$, the number of avoiding
subsets of $\mathbb{Z}_{p}$ of even positive size $s\leq \frac{p-1}{2}$ is $%
\binom{\left( p-1\right) /2-s/2}{s/2}$.
\end{proposition}

\begin{proof}
We have $s\geq 2$ and even. Every set as in Equation (\ref{Eq_Avoiding}) can
be represented by a characteristic vector $\underline{c}\in \left\{
0,1\right\} ^{\left( p+1\right) /2}$, indexed by $\left\{ 0,1,2,...,\left(
p-1\right) /2\right\} $, where $c_{j}=1$ if and only if $j\in \left\{
i_{1},i_{2},...,i_{s/2}\right\} $. The condition $i_{j+1}\geq i_{j}+2$ for
all $1\leq j\leq \frac{s}{2}$, where $i_{\frac{s}{2}+1}:=p-i_{\frac{s}{2}}$,
translates to the condition that $\underline{c}$ starts and ends with $0$,
and no two of the $s/2$ entries which are equal to $1$ are consecutive.
Hence, the number of such distinct vectors $\underline{c}$ is equal to the
number of ways to distribute $\left( p+1\right) /2-s/2$ identical balls into 
$s/2+1$ bins so that none of the bins is left empty. This number is
precisely $\binom{\left( p-1\right) /2-s/2}{s/2}$.
\end{proof}

\begin{corollary}
For a prime $p$, the total number of avoiding subsets of $\mathbb{Z}_{p}$ is
bounded above by is $2^{p\left( c+o\left( 1\right) \right) }$, for $c\approx
0.3471$.
\end{corollary}

\begin{proof}
It is well known (see, e.g., \cite[Section 22.5]{JuknaBook}) that for every
positive integer $n$ and real number $\alpha \in \left[ 0,1\right] $ (here $%
\left[ 0,1\right] $ stands for the real interval),%
\begin{equation*}
\binom{n}{\alpha n}\leq 2^{n\cdot H\left( \alpha \right) }\text{,}
\end{equation*}%
where $H:\left[ 0,1\right] \rightarrow \left[ 0,1\right] $ is the binary
entropy function defined by%
\begin{equation*}
H\left( \alpha \right) =-\alpha \cdot \log _{2}\alpha -\left( 1-\alpha
\right) \cdot \log _{2}\left( 1-\alpha \right) \text{, }\forall \alpha \in %
\left[ 0,1\right] \text{.}
\end{equation*}

We apply this to derive an upper bound on $\binom{\left( p-1\right) /2-s/2}{%
s/2}$, which, by Proposition \ref{Prop_NumOfAvoiding}, is the number of
avoiding sets of positive size $s\leq \frac{p-1}{2}$. Define $\beta :=s/p$
and take $n:=\left( p-1\right) /2-s/2=p\left( 1-\beta \right) /2-\left(
1/2\right) $. Then $0<\beta \leq 1/2$, and defining the real parameter $%
\alpha $ by $s/2=\alpha n$, we get 
\begin{equation*}
\alpha =\alpha n/n=\frac{\beta p/2}{p\left( 1-\beta \right) /2-\left(
1/2\right) }=\frac{\beta }{1-\beta }\left( 1+\Theta \left( 1/p\right)
\right) \text{.}
\end{equation*}%
Therefore $\binom{\left( p-1\right) /2-s/2}{s/2}\leq $ $2^{p\left( \frac{%
1-\beta }{2}\cdot H\left( \frac{\beta }{1-\beta }\right) +o\left( 1\right)
\right) }$. Let $c$ denote the maximum of $\frac{1-\beta }{2}\cdot H\left( 
\frac{\beta }{1-\beta }\right) $ for $\beta $ in the real interval $\left[
0,1/2\right] $. One can verify that this maximum is attained for $\beta
\approx 0.276$. Summing over all positive even $s\leq \frac{p-1}{2}$ gives
the claim of the corollary.
\end{proof}

\begin{remark}
Let $p>5$ be a prime. Note that by the extra condition following Equation (%
\ref{Eq_pi1_2}) we have to rule out $S\subseteq \mathbb{Z}_{p}^{\ast }$ for
which $S+S=\mathbb{Z}_{p}$ and $S_{1}:=\overline{S\cup \left\{ 0\right\} }$
is an arithmetic progression of size $\frac{p-1}{2}$. Since $S_{1}$ is a
symmetric set which does not contain $0$ we must have that $\left\vert
S_{1}\right\vert =\frac{p-1}{2}$ is even. This implies $p\equiv 1\left( 
\func{mod}4\right) $. Furthermore, up to an automorphism, $S_{1}=\left[ 
\frac{p+3}{4},\frac{3p-3}{4}\right] $. The reader can check that 
\begin{equation*}
S=\overline{S_{1}\cup \left\{ 0\right\} }=\left[ 1,\tfrac{p-1}{4}\right]
\cup \left[ \tfrac{3p+1}{4},p-1\right]
\end{equation*}%
satisfies $\left\vert S\right\vert =\frac{p-1}{2}$ and $S+S=\mathbb{Z}_{p}$
and hence should be excluded, together with all of its automorphic images,
from the set of solutions of Equation (\ref{Eq_pi1_2}).
\end{remark}

\subsubsection{Solutions to Equation (\protect\ref{Eq_pi1_3}) \label%
{Subsect_SolutionsTo3}}

The following theorem characterizes the solutions to Equation (\ref{Eq_pi1_3}%
).

\begin{theorem}
Let $G=\left( \mathbb{Z}_{p},+_{p}\right) $, where $p>5$ is a prime. The
solutions of Equation (\ref{Eq_pi1_3}) are all the subsets $S\subseteq G$
which satisfy: $\left\vert S\right\vert =\frac{p-1}{2}$, $S\cap \left(
-S\right) =\emptyset $ and $S$ is not an arithmetic progression. The number
of d-partitions of $G$ defined by these solutions is $2^{\frac{p-1}{2}-1}-%
\frac{p-1}{2}$.
\end{theorem}

\begin{proof}
1. Suppose that $S\subseteq G$ satisfies $\left\vert S\right\vert =\frac{p-1%
}{2}$, $S\cap \left( -S\right) =\emptyset $ and that $S$ is not an
arithmetic progression. We will prove that $S$ is a solution to Equation (%
\ref{Eq_pi1_3}). The condition $S\cap \left( -S\right) =\emptyset $ implies
that for any $x,y\in S$ we have $x+y\neq 0$, and hence $0\notin $ $S+S$.
Since $\left\vert S\right\vert =\frac{p-1}{2}$, we have, by Theorem \ref%
{Th_CauchyDavenport} that $\left\vert S+S\right\vert \geq 2\left\vert
S\right\vert -1=p-2$. Thus $p-2\leq \left\vert S+S\right\vert \leq p-1$.
However, if $p-2=\left\vert S+S\right\vert $ then $S$ is an arithmetic
progression by Theorem \ref{Th_Vosper} - a contradiction. Therefore $%
\left\vert S+S\right\vert =p-1$, which together with $0\notin $ $S+S$
implies $S+S=\overline{\left\{ 0\right\} }$. Furthermore, $S\cap \left(
-S\right) =\emptyset $ implies $0\notin S$, and so $S$, $-S$ and $\left\{
0\right\} $ are mutually disjoint. Since $\left\vert S\right\vert =\frac{p-1%
}{2}$ we get $S=-\overline{S\cup \left\{ 0\right\} }$.

2. Suppose that $S$ is a solution to Equation (\ref{Eq_pi1_3}). We will
prove $\left\vert S\right\vert =\frac{p-1}{2}$, $S\cap \left( -S\right)
=\emptyset $ and that $S$ is not an arithmetic progression. Since $S=-%
\overline{S\cup \left\{ 0\right\} }$ we have that $0\notin S$, so $G=S\cup
\left( -S\right) \cup \left\{ 0\right\} $ is a disjoint union, and hence $%
S\cap \left( -S\right) =\emptyset $ and $\left\vert S\right\vert =\frac{p-1}{%
2}$. Finally note that $S$ cannot be an arithmetic progression, since, in
that case, $\left\vert S+S\right\vert =2\left\vert S\right\vert -1=$ $p-2$
while $S+S=\overline{\left\{ 0\right\} }$ implies $\left\vert S+S\right\vert
=p-1$ - a contradiction.

3. It remains to count the number of d-partitions defined by the solutions
of Equation (\ref{Eq_pi1_3}). The number of sets $S\subseteq G$ which
satisfy $\left\vert S\right\vert =\frac{p-1}{2}$ and $S\cap \left( -S\right)
=\emptyset $ is $2^{\frac{p-1}{2}}$. From this number we have to subtract
the number of arithmetic progressions of size $\tfrac{p-1}{2}$ which satisfy 
$S\cap \left( -S\right) =\emptyset $. Observe that $S\subseteq G$ is an
arithmetic progression of size $\tfrac{p-1}{2}$ which satisfies $S\cap
\left( -S\right) =\emptyset $, if and only if for any $\delta \in \mathbb{Z}%
_{p}^{\ast }$, $\delta \cdot _{p}S$ is also an arithmetic progression with
the same properties. Let $S$ be an arithmetic progression such that $%
\left\vert S\right\vert =\frac{p-1}{2}$ and $S\cap \left( -S\right)
=\emptyset $ . Let $\delta $ be the step of $S$. Then $\delta \in \mathbb{Z}%
_{p}^{\ast }$ and $\delta ^{-1}\cdot _{p}S$ is an arithmetic progression of
step $1$ such that $\left\vert \delta ^{-1}\cdot _{p}S\right\vert =\frac{p-1%
}{2}$ and $\left( \delta ^{-1}\cdot _{p}S\right) \cap \left( -\delta
^{-1}\cdot _{p}S\right) =\emptyset $ . But one can easily check that this
implies $\delta ^{-1}\cdot _{p}S=$ $\left[ 1,\tfrac{p-1}{2}\right] $.
Therefore, the number of solutions to Equation (\ref{Eq_pi1_3}) is $2^{\frac{%
p-1}{2}}-\left( p-1\right) $. Finally notice that the solutions to Equation (%
\ref{Eq_pi1_3}) divide into pairs $\left( S,-S\right) $ where $S$ and $-S$
define the same d-partition. This proves the claim of the theorem.
\end{proof}

\section{Appendix\label{Appendix_ProofOdDfield}}

\begin{proof}[Proof of Proposition \protect\ref{Prop_d-field}]
We replace $\oplus $ by $+$ and write $xy$ for $x\otimes y$. Let $d\in
D\backslash \left\{ \varepsilon \right\} $, and consider the infinite
sequence of sums, $d,d+d,d+d+d,...$. By definition of the canonical order $%
\leq _{D}$, any two successive terms in the sequence, with $n$ and $n+1$
summands ($n\geq 1$ an integer), are either equal or satisfy 
\begin{equation*}
\underset{n\text{ summands}}{\underbrace{d+d+\cdots +d}}~<_{D}~\underset{n+1%
\text{ summands}}{\underbrace{d+d+\cdots +d}}\text{.}
\end{equation*}%
Thus, if no two successive terms of the sequence are equal, the sequence is
strictly increasing with respect to $\leq _{D}$, and hence, its terms are
pairwise distinct. Otherwise, let $n$ be the smallest natural number such
that 
\begin{equation*}
\underset{n\text{ summands}}{\underbrace{d+d+\cdots +d}}~=~\underset{n+1%
\text{ summands}}{\underbrace{d+d+\cdots +d}}\text{.}
\end{equation*}%
Then the sequence of sums is strictly increasing up to the $n$-fold sum and,
as can be proven by easy induction, any $m$-fold sum with $m>n$ is equal to
the $n$-fold sum. In such a case we say that $d$ is $n$-idempotent (\cite[p.
15]{GondranMinouxDioidBook2010}). Using distributivity we get:%
\begin{equation*}
\underset{n\text{ summands}}{\underbrace{\left( e+e+\cdots +e\right) }d}~=~%
\underset{n+1\text{ summands}}{\underbrace{\left( e+e+\cdots +e\right) }d}%
\text{.}
\end{equation*}%
Since $d\neq \varepsilon $ it is invertible and we get:%
\begin{equation*}
\underset{n\text{ summands}}{\underbrace{e+e+\cdots +e}}~=~\underset{n+1%
\text{ summands}}{\underbrace{e+e+\cdots +e}}\text{.}
\end{equation*}%
Multiplying by any $d^{\prime }\in D\backslash \left\{ \varepsilon \right\} $
we obtain 
\begin{equation*}
\underset{n\text{ summands}}{\underbrace{d^{\prime }+d^{\prime }+\cdots
+d^{\prime }}}~=~\underset{n+1\text{ summands}}{\underbrace{d^{\prime
}+d^{\prime }+\cdots +d^{\prime }}}\text{.}
\end{equation*}%
It follows that either for each $d\in D\backslash \left\{ \varepsilon
\right\} $ all finite sums of the form $d,d+d,d+d+d,...$ are distinct or
there exists a unique minimal $n$ such that each $d\in D\backslash \left\{
\varepsilon \right\} $ is $n$-idempotent. Furthermore, assume that each $%
d\in D\backslash \left\{ \varepsilon \right\} $ is $n$-idempotent, where $%
n\geq 1$. Consider 
\begin{equation*}
x:=~\underset{n\text{ summands}}{\underbrace{e+e+\cdots +e}}\text{.}
\end{equation*}%
Computing $x^{2}$ using the distributive law gives 
\begin{equation*}
x^{2}=\underset{n^{2}\text{ summands}}{~\underbrace{e+e+\cdots +e}}~=~%
\underset{n\text{ summands}}{\underbrace{e+e+\cdots +e}}\text{.}
\end{equation*}%
Hence $x^{2}=x$. Since $x$ is invertible this immediately implies $x=e$.
Thus we are forced to have $n=1$, which is equivalent to $D$ being an
idempotent d-field.

The last part of the proof is based on \cite{Gunawardena} (Proposition 2.7
and its proof). Let $D$ be an idempotent d-field. Suppose that $D$ has a
largest element $g$. Note that this assumption holds true if $D$ is finite
since in such a case we can simply take $g$ to be the sum of all elements of 
$D$. By assumption, $e\leq _{D}g$ which gives, when multiplied by $g$, $%
g\leq _{D}g^{2}$ (\cite[Proposition 1.6.1.7]{GondranMinouxDioidBook2010}).
Since $g$ is the largest element this forces $g^{2}=g$ and hence, since $g$
is invertible, $g=e$. Thus $e$ is the greatest element of $D$. Let $d\in
D\backslash \left\{ \varepsilon \right\} $. Then $d$ is invertible and $%
d^{-1}\leq e$ since $e$ is the greatest element of $D$. Multiplying by $d$
gives $e\leq d$ which forces $d=e$. This proves that $D=\mathbb{B}$.
\end{proof}

\textbf{{}Acknowledgements: }We would like to thank Marcel Herzog, Mikhail
Klin, Attila Mar\'{o}ti, Andrew Misseldine and Mikhail Muzychuk for useful
discussions at various stages of this work. We are grateful to the anonymous
referees for their comments and especially for bringing Tamaschke's work to
our attention.

\bigskip

\bigskip

\bigskip

\bigskip


\begin{thebibliography}{99}
\bibitem{BaldanGadducci2008} P. Baldan, F. Gadducci, Petri Nets Are Dioids.
In: Meseguer J., Ro\c{s}u G. (eds) Algebraic Methodology and Software
Technology. AMAST 2008. Lecture Notes in Computer Science, vol 5140.
Springer, Berlin, Heidelberg (2008).

\bibitem{CameronSumFreeSurvey1987} P. J. Cameron, Portrait of a typical
sum-free set, In C. Whitehead, editor, Surveys in Combinatorics 1987, vol.
123, Cambridge University Press,\ (1987), p.13-42.

\bibitem{Cauchy1813} A. Cauchy. Recherches sur les nombres. J. E%
\'{}%
cole Polytech, 40:99--116, 1813.

\bibitem{CSX2014} J. Campos, C. Seatzu and X. Xie, Formal Methods in
Manufacturing, CRC press, 2014.

\bibitem{Davenport1935} H. Davenport. On the addition of residue classes. J.
London Math. Soc., 10:30--32, 1935.

\bibitem{DelCampoEtAl2015} C. del-Campo, C. Pel\'{a}ez-Moreno, F. J.
Valverde-Albacete, Activating Generalized Fuzzy Implications from Galois
Connections. In: Magdalena, L., Verdegay, J. L. \& Esteva, F. (eds.). Enric
Trillas: A Passion for Fuzzy Sets A Collection of Recent Works on Fuzzy
Logic. (pp. 201-212). (Studies in Fuzziness and Soft Computing; 322).
Springer International Publishing, 2015.

\bibitem{CGLMS2014conjugates} J.~Cannon, M.~Garonzi, D.~Levy, A.~Mar\'{o}ti,
and I.~Simion. \newblock Groups equal to a product of three conjugate
subgroups. \newblock Israel J. Math., (2016), Volume 215, Issue 1, pp 31--52.

\bibitem{DeshouillersFreiman2006} J.-M. Deshouillers and G. A. Freiman. On
sum-free sets modulo p. Funct. Approx. Comment. Math., 35(1):51--59, 2006.

\bibitem{DeshouillersLev2008} J.-M. Deshouillers and V. F. Lev. A refined
bound for sum-free sets in groups of prime order. Bull. Lond. Math. Soc.,
40(5):863--875, 2008.

\bibitem{GondranMinouxFuzzy2007} M. Gondran and M. Minoux, Dioids and
Semirings: Links to fuzzy sets and other applications, Fuzzy Sets and
Systems, 158(12): 1273-1294, 2007.

\bibitem{GondranMinouxDioidBook2010} M. Gondran and M. Minoux, "Graphs,
Dioids and Semirings: New Models and Algorithms". Operations
Research/Computer Science Interfaces Series, Springer, 2010.

\bibitem{BGordonPacific1964} B. Gordon. A generalization of the coset
decomposition of a finite group, Pacif. J. Math. 15 (1965), 503-509.

\bibitem{GreenRuzsa2005} B.J. Green and I.Z. Ruzsa, Sum-free sets in abelian
groups, Israel J. Math., 147 (2005), 157-188.

\bibitem{Gunawardena} J. Gunawardena, An introduction to idempotency, In
Gunawardena, Jeremy. Idempotency. Based on a workshop, Bristol, UK, October
3--7, 1994. Cambridge: Cambridge University Press. pp. 1--49.

\bibitem{HavivLevySF2017} I. Haviv and D. Levy, Symmetric complete sum-free
sets in cyclic groups, Electronic Notes in Discrete Mathematics (EuroComb),
61 (2017), 585--591, (To appear in Israel J. of Math).

\bibitem{JuknaBook} S. Jukna, Extremal Combinatorics: With Applications in
Computer Science, Springer-Verlag, Texts in theoretical computer science,
Second edition, 2011.

\bibitem{KozmaLev1992} G. Kozma and A. Lev, Bases and decomposition numbers
for finite groups, Arch. Math. 58 (1992), 417---424.

\bibitem{MuzychukKlinPoschel2001} M. Muzychuk, M. Klin, R. P\"{o}schel, The
isomorphism problem for circulant graphs via Schur ring theory, DIMACS Ser.
Discrete Math. Theoret. Comput. Sci. 56 (2001) 241-264.

\bibitem{MuzychukPonomarenko2009} M.Muzychuk, I.Ponomarenko, Schur rings,
European Journal of Combinatorics 30 (2009) 1526-1539.

\bibitem{RhemtullaStreet} A.H. Rhemtulla, A.P. Street, Maximal sum-free sets
in finite abelian groups, Bull. Austral. Math. Soc., vol. 2 (1970), 289-297.

\bibitem{Schur1916} I. Schur. \"{U}ber die kongruenz $x^{m}+y^{m}\equiv
z^{m}(mod~p)$. Jahresbericht der DeutschenMathematiker-Vereinigung,
25:114--117, 1916.

\bibitem{Schur1933} I. Schur, Zur Theorie der einlach transitiven
Permutationgruppen, S.-B. Preuss. Akad. Wiss. phys.-math. Kl. 18/20(1933),
598-623.

\bibitem{StreetWhitehead1974} A.P. Street, E.G. Whitehead. Group Ramsey
Theory, Journal of combinatorial Theory (A), 17, (1974), 219-226.

\bibitem{Tamaschke1968} O. Tamaschke, An extension of group theory to
S-semigroups. Math. Z. 104, 74-90 (1968).

\bibitem{Tamaschke1969} O. Tamaschke, On the theory of Schur rings, Annali
di Matematica 81(1), 1-43, (1969).

\bibitem{TaoVuAdditive2007} T. Tao, V. Vu, Additive Combinatorics, Cambridge
studies in advanced mathematics 105, Cambridge University Press, (2007).

\bibitem{Vosper1956} A.G. Vosper, The critical pairs of subsets of a group
of prime order, J. London Math. Soc 31 (1956), 200-205.

\bibitem{WielandtPermGroups1964} H. Wielandt, Finite Permutation Groups,
Academic press, New York, London, 1964.

\bibitem{Yap1968} H.P. Yap, The number of maximal sum-free sets in $C_{p}$,
Nanta Mathematica, 2 (1968), 68-71.

\bibitem{Yap1970} H.P. Yap, Structure of maximal sum-free sets in $C_{p}$,
Acta Arithmetica 17 (1970), 29-35.
\end{thebibliography}
\end{document}